\documentclass[11pt]{article}
\usepackage[top=2.5cm, bottom=2.5cm, left=2.5cm, right=2.5cm] {geometry}

\usepackage{comment}
\usepackage{amsmath,amssymb,theorem,color}
\numberwithin{equation}{section}
 \usepackage{amssymb}

\usepackage{color}
\usepackage{amsmath}
\usepackage{graphicx}
\usepackage{amsfonts}%
\newcommand{\R}{\ensuremath{\mathbb{R}}}
\newcommand{\Q}{\ensuremath{\mathbb{Q}}}
\newcommand{\N}{\ensuremath{\mathbb{N}}}

\newcommand{\cH}{\mathcal{H}}
\newcommand{\cC}{\mathcal{C}}
\newcommand{\cF}{\mathcal{F}}
\newcommand{\bP}{\mathbb{P}}

\newcommand{\Z}{\mathbb{Z}}

\newcommand{\ltn}{\ensuremath{\left| \! \left| \! \left|}}
\newcommand{\rtn}{\ensuremath{\right| \! \right| \! \right|}}

\newtheorem{theorem}{Theorem}[section]
{ \theorembodyfont{\normalfont}

\newtheorem{remark}[theorem]{Remark}
}

\newtheorem{definition}[theorem]{Definition}
\newtheorem{lemma}[theorem]{Lemma}

\newtheorem{proposition}[theorem]{Proposition}

\newcounter{enumctr}

\begin{document}
\title{Lyapunov spectrum of nonautonomous linear Young differential equations}
\author{
Nguyen Dinh Cong\thanks{Institute of Mathematics, Vietnam Academy of Science and Technology, Vietnam {\it E-mail: ndcong@math.ac.vn}}, $\;$ Luu Hoang Duc\thanks{Institute of Mathematics, Vietnam Academy of Science and Technology, \& Max-Planck-Institut f\"ur Mathematik in den Naturwissenschaften, Leipzig, Germany {\it E-mail: lhduc@math.ac.vn, duc.luu@mis.mpg.de}}, $\;$ Phan Thanh Hong \thanks{Thang Long University, Hanoi, Vietnam {\it E-mail: hongpt@thanglong.edu.vn }}\\[2ex]{\it in memory of V. M. Millionshchikov}}
\date{}
\maketitle

\begin{abstract}
	We show that a linear Young differential equation generates a topological two-parameter flow, thus the notions of Lyapunov exponents and Lyapunov spectrum are well-defined. The spectrum can be computed using the discretized flow and is independent of the driving path for triangular systems which are regular in the sense of Lyapunov. In the stochastic setting, the system generates a stochastic two-parameter flow which satisfies the integrability condition, hence the Lyapunov exponents are random variables of finite moments. Finally, we prove a Millionshchikov theorem stating that almost all, in a sense of an invariant measure, linear nonautonomous Young differential equations are Lyapunov regular. 
\end{abstract}


\section{Introduction}
In this article we study the Lyapunov spectrum of the nonautonomous linear Young differential equation (abbreviated by YDE)
\begin{equation}\label{lin1}
dx(t) = A(t)x(t) dt + C(t) x(t) d \omega(t),\ x(t_0)=x_0 \in \R^d, t\geq t_0,
\end{equation}
where $A, C$ are continuous matrix valued functions on $[0,\infty)$, and $\omega$ is a continuous path on $[0,\infty)$ having finite $p$-th variation on each compact interval of $[0,\infty)$, for some $p \in (1,2)$. Such system \eqref{lin1} appears, for instance, when considering the linearization of the autonomous Young differential equation 
\begin{equation}\label{SDE}
dy(t) = f(y(t))dt + g(y(t))d\omega(t)
\end{equation}
along any reference solution $y(t, y_0,\omega)$. An example is when we would like to solve in the pathwise sense stochastic differential equations driven by fractional Brownian motions with Hurst index $H \in (\frac{1}{2},1)$ defined on a complete probability space $(\Omega,{\cF},{\bP})$ \cite{nualart3}. In fact it follows from \cite{cassetal} that \eqref{SDE} under the stochastic setting also satisfies the integrability condition.

The equation \eqref{lin1} can be rewritten in the integral form
\begin{equation}\label{lin2}
x(t) = x_0 + \int_{t_0}^t A(s)x(s)ds + \int_{t_0}^t C(s)x(s)d\omega(s), \ t\geq t_0,
\end{equation} 
where the second integral is understood in the Young sense \cite{young}, which can also be presented in terms of fractional derivatives \cite{zaehle}. Under some mild conditions, the unique solution of \eqref{lin1} generates a two-parameter flow $\Phi_\omega(t_0,t)$, as seen in \cite{congduchong17}. Under a certain stochastic setting, \eqref{lin1} actually generates a stochastic two-parameter flow in the sense of Kunita \cite{kunita}.

Our aim is to study the Lyapunov exponents and Lyapunov spectrum of the linear two-parameter flow generated from Young equation \eqref{lin1}. Notice that Lyapunov spectrums and its splitting are the main content of the celebrated multiplicative ergodic theorem (MET) by Oseledets \cite{Ose68}. It was also investigated by Millionshchikov in \cite{Mil86,Mil87,mil68,Mil68} for linear nonautonomous differential equations. In the stochastic setting, the MET is also formulated for random dynamical systems in \cite[Chapter 3]{arnold}. Further investigations can be found in \cite{Cong1,cong,cong2004, congquynh09} for stochastic flows generated by nonautonomous linear stochastic differential equations driven by standard Brownian motion. \\
For Young equations, we show that Lyapunov exponents can be computed based on the discretization scheme. Moreover, if the driving path $\omega$ satisfies certain conditions, the Lyapunov spectrum can be computed independently of $\omega$ for {\it triangular systems} (i.e. both $A, C$ are upper triangular matrices) which are Lyapunov regular. 

One important issue is the non-randomness of Lyapunov exponents when the system is considered under a certain stochastic setting, namely if the driving path $\omega$ is a realization of a certain stochastic noise. In case the system is driven by standard Brownian noises, a filtration of independent $\sigma-$ algebras can be constructed and the argument of Kolmogorov's zero-one law can be applied to prove the non-randomness of Lyapunov exponents, which are measurable to tail events, see \cite{cong}. \\
In general, the stochastic noise might be a fractional Brownian motion which is not Markov, hence it is difficult to construct such a filtration and to apply the Kolmogorov's zero-one law. The question of non-randomness of Lyapunov spectrum is therefore still open. However, the answer is affirmative for some special cases. For example, autonomous and periodic systems can generate random dynamical systems satisfying the integrability condition, thus the Lyapunov spectrum is non-random by the multiplicative ergodic theorem \cite{arnold}. Our investigation shows that the Lyapunov spectrum of triangular systems that are Lyapunov regular are also non-random. In general, we expect that the statement of non-randomness of Lyapunov spectrum is still true for any Lyapunov regular system, although finding a counter-example of a nonautonomous system with random Lyapunov spectrum also attracts our interest.

The paper is organized as follow. In Section 2, we prove in Proposition  \ref{thm.lin.flow} the generation of a two-parameter flow from the unique solution of \eqref{lin1}. The concepts of Lyapunov exponents and Lyapunov spectrum of system \eqref{lin1} are then defined in Section 3. Under the assumptions on the driving path $\omega$ and on the coefficient functions, we prove in Theorem \ref{thm2} that Lyapunov spectrum can be computed using the discretized flow and give an explicit formula of the spectrum in Theorem \ref{triangularcase} in case of triangular systems which are regular in the sense of Lyapunov.  Theorem \ref{thm.L.reg.triangle} provides a criterion for a triangular system of YDE to be Lyapunov regular.
In Section 4, we consider the system under random perspectives in which the driving path acts as a realization of a stochastic stationary process in a function space equipped with a probabilistic framework. The system is then proved to generate a stochastic two-parameter flow which satisfies the integrability condition, hence the Lyapunov exponents are proved in Theorem \ref{LEbounded} to be random variables of finite moments.  Subsection 4.2 is devoted to study the regularity of the system, where we prove a Millionshchikov Theorem \ref{thm.Millionshchikov} stating that almost all, in a sense of an invariant measure, nonautonomous linear Young differential equations are Lyapunov regular. We end up with a discussion on the non-randomness of Lyapunov spectrum in some special cases, and raise this interesting question in general.

\section{Preliminary}

In this section we present some well-known facts of Young differential equations and two parameter flows. Let $0 \leq T_1 < T_2 < \infty$. Denote by $\cC([T_1,T_2], \R^{d\times d})$ the Banach space of continuous matrix-valued functions on $[T_1,T_2]$ equipped with the sup norm $\|\cdot\|_{\infty,[T_1,T_2]}$, by $\cC^{r{\rm-var}}([T_1,T_2],\R^d)$ the Banach space of bounded $r-$variation continuous functions on $[T_1,T_2]$ having values in $\R^d$ with the norm
 \begin{equation*}
 	\|u\|_{r{\rm-var},[T_1,T_2]}=|u(T_1)|+\ltn u \rtn_{r{\rm-var},[T_1,T_2]} <\infty,
 \end{equation*}
in which $|\cdot|$ is the Euclidean norm and $\ltn \cdot\rtn_{r{\rm-var},[T_1,T_2]}$ is the seminorm defined by 
 $$
 \ltn u\rtn_{r{\rm-var},[T_1,T_2]} = \left(\sup_{\Pi(T_1,T_2)}\sum_{i=0}^{n-1}|u(t_{i+1})-u(t_i)|^r\right)^{1/r},\quad  u\in \cC^{r\rm{-}\rm{var}}([T_1,T_2],\R^d),
 $$
 where the supremum is taken over the whole class of finite partition $\Pi (T_1,T_2)=\{ T_1= t_0<t_1<\cdots < t_n=T_2\}$ of $[T_1,T_2]$.  For each $0<\alpha<1$, we denote by $\cC^{\alpha\rm{-Hol}}([T_1,T_2],\R^d)$ the space of $\alpha-$H\"older continuous functions on $[T_1,T_2]$ equipped with the norm
 $$\|u\|_{\alpha,[T_1,T_2]}: =  \|u\|_{\infty,[T_1,T_2]} +\ltn u \rtn_{\alpha{\rm-Hol},[T_1,T_2]} ,$$ %|u(T_1)| + \sup_{T_1\leq s<t\leq T_2}\frac{|u(t)-u(s)|}{(t-s)^\alpha}.$$
 in which $\|u \|_{\infty,[T_1,T_2]}:= \sup_{t\in[T_1,T_2]}|u(t)|$  and $\ltn u \rtn_{\alpha{\rm-Hol},[T_1,T_2]} =  \sup_{T_1\leq s<t\leq T_2}\frac{|u(t)-u(s)|}{(t-s)^\alpha}$. It is obvious that for all $u\in \cC^{\alpha\rm{-Hol}}([T_1,T_2],\R^d)$,
 $$\ltn u \rtn_{r{\rm-var},[T_1,T_2]}\leq (T_2-T_1)^\alpha\ltn u \rtn_{\alpha{\rm-Hol},[T_1,T_2]},$$
 with $\alpha=1/r$. Moreover, we have the following estimate, whose proof follows directly from the definitions of the $p-$var seminorm and the sup norm and will be omitted here. 
 \begin{lemma}\label{qnorm-product}
 Let $t_0\geq 0$ and $T>0$ be arbitrary.
 	If $C\in \cC^{q{\rm-var}}([t_0,t_0+T],\R^{d\times d})$, $ x\in \cC^{q\rm{{\rm-var}}}([t_0,t_0+T],\R^{d})$, then for all $s<t$ in $[t_0,t_0+T]$,
 	\begin{eqnarray*}
 		\ltn Cx\rtn_{q{\rm-var},[s,t]}\leq \|C\|_{\infty,[s,t]}\ltn x\rtn_{q{\rm-var},[s,t]} + \|x\|_{\infty,[s,t]}\ltn C\rtn_{q{\rm-var},[s,t]}.
 	\end{eqnarray*}
 \end{lemma}
 Now, consider $x\in \cC^{q\rm{-}\rm{var}}([T_1,T_2],\R^{d\times m})$ and $\omega\in \cC^{p\rm{-}\rm{var}}([T_1,T_2],\R^m)$, $p,q \geq 1$ and $\frac{1}{p}+\frac{1}{q}  > 1$, the Young integral $\int_a^bx(t)d\omega(t)$ can be defined as 
 \[
 \int_a^bx(t)d\omega(t):= \lim \limits_{|\Pi| \to 0} \sum_{t_i\in \Pi} x(t_i)(\omega(t_{i+1})-\omega(t_i)), 
 \]
 where the limit is taken on all finite partitions $\Pi = \{T_1 = t_0 < t_1 < \ldots < t_n = T_2\}$ with $|\Pi| := \displaystyle\max_{0\leq i \leq n-1} |t_{i+1}-t_i|$ (see \cite[p. 264--265]{young}). This integral satisfies additive property by the construction, and the so-called Young-Loeve estimate \cite[Theorem 6.8, p. 116]{friz}
 \begin{equation}\label{YL0}
 	\Big|\int_s^t x(u)d\omega(u)-x(s)[\omega(t)-\omega(s)]\Big| \leq K \ltn x\rtn_{q{\rm-var},[s,t]} \ltn\omega\rtn_{p{\rm-var},[s,t]},
 \end{equation}
 where 
 \begin{equation}\label{constK}
 	K:=(1-2^{1-\theta})^{-1},\qquad \theta := \frac{1}{p} + \frac{1}{q} >1.
 \end{equation}

%%%%%%%%%%%%%%%%%%%%%%%%%%%%%%%%%%%%%%%%%%%%

Now for any $\omega\in \cC^{p{\rm-var}}([t_0,t_0+T],\R)$ with some $1< p<2 $, we consider the deterministic Young equation
\begin{equation}\label{linearfSDE}
x(t)=x_0 + \int_{t_0}^tA(s)x(s)ds+\int_{t_0}^tC(s)x(s)d\omega(s), 
\end{equation}
in which $A\in \cC([t_0,t_0+T],\R^{d\times d}),C\in \cC^{q{\rm-var}}([t_0,t_0+T],\R^{d\times d})$ with $ q > p$ and $\frac{1}{q}+\frac{1}{p}>1$. We first show that under mild conditions on coefficient functions $A,C$, \eqref{linearfSDE} has a unique solution in $ \cC^{q{\rm-var}}([t_0,t_0+T],\R^{d})$. 

\begin{proposition}\label{existence}
	Fix $[t_0,t_0+T]$ and consider $\omega$ varying as an element of the Banach space $\cC^{p{\rm-var}}([t_0,t_0+T])$. Assume that $A\in \cC([t_0,t_0+T],\R^{d\times d}),C\in \cC^{q{\rm-var}}([t_0,t_0+T],\R^{d\times d})$ with $q>p$ and $\frac{1}{q}+\frac{1}{p}>1$. Then equation \eqref{linearfSDE} has a unique solution $x (\cdot,t_0,x_{0},\omega)$  in the space $\cC^{p{\rm-var}}([t_0,t_0+T],\R^d)$ which satisfies
	\begin{alignat}{2}
		&(i)\;\; \|x(\cdot,t_0,x_0,\omega)\|_{\infty,[t_0,t_0+T]}\leq |x_0| \exp \Big\{\eta[2+ (\frac{2M^*}{\mu})^p(T^p+\ltn \omega\rtn^p_{p{\rm-var},[t_0,t_0+T]})]\Big\},\label{growth}\\
		&(ii)\;\; \ltn x(\cdot,t_0,x_0,\omega)\rtn_{p{\rm-var},[t_0,t_0+T]}\leq |x_0| \exp \Big\{(1+\eta)[3+ (\frac{2M^*}{\mu})^p(T^p+\ltn \omega\rtn^p_{p{\rm-var},[t_0,t_0+T]})]\Big\}\label{growth1}
	\end{alignat}
	where 
	\begin{equation}\label{M*}
		M^*=M^*(t_0,T):=\max\{\|A\|_{\infty,[t_0,t_0+T]}, 2K \| C\|_{q{\rm-var},[t_0,t_0+T]}\}<\infty,
	\end{equation}
	$K$ is defined in \eqref{constK}, $\mu$ is a constant such that $0<\mu<\min\{1,M^*\}$ and $\eta= -\log(1-\mu)$. In addition, the solution mapping
	\begin{eqnarray*}
		%X: [t_0,t_0+T]\times \R^d\times \cC^{p{\rm-var}}([t_0,t_0+T],\R)&\longrightarrow & \cC^{p{\rm-var}}([t_0,t_0+T],\R^d) \\
		%(a,x_0,\omega) &\mapsto& x(\cdot,a,x_0,\omega)
		X:  \R^d\times \cC^{p{\rm-var}}([t_0,t_0+T],\R)&\longrightarrow & \cC^{p{\rm-var}}([t_0,t_0+T],\R^d) \\
		(x_0,\omega) &\mapsto& x(\cdot,t_0,x_0,\omega).
	\end{eqnarray*}
	is continuous w.r.t $(x_0,\omega)$. 
\end{proposition}
\begin{proof}
See the appendix.
\end{proof}

\begin{remark}
(i) Fix $[t_0,t_0+T]$, by considering the backward equation similar to that of \cite{congduchong17}, we can draw the same conclusions on the existence and uniqueness of the solution for the backward equation at an arbitrary point $a\in [t_0,t_0+T]$. Moreover, it can be proved that the solution mapping $X$ is continuous with respect to $(a,x_0,\omega) \in [t_0,t_0+T]\times \R^d\times \cC^{p{\rm-var}}([t_0,t_0+T],\R)$.\\
(ii) If $\omega\in \cC^{1/p\rm{-Hol}}([t_0,t_0+T],\R)\subset  \cC^{p{\rm-var}}([t_0,t_0+T],\R)$ then similar arguments prove that the solution is $1/p-$H\"older continuous and the solution mapping $X$ is continuous with respect to $(a,x_0,\omega) \in [t_0,t_0+T]\times \R^d\times \cC^{1/p\rm{-Hol}}([t_0,t_0+T],\R)$.
\end{remark}

For any $t_0\leq t_1\leq t_2\leq t_0+ T$ the {\em Cauchy operator} $\Phi_\omega(t_1,t_2): \R^d \rightarrow \R^d$ of the YDE \eqref{lin1} is defined as $\Phi_{\omega}(t_1,t_2)x_{t_1}:= x(t_2,t_1,x_{t_1},\omega)$ for any vector $x_{t_1}\in \R^d$.

Following \cite[p. 551]{arnold}, a family of mappings $X_{s,t} : \R^d \rightarrow \R^d$ depending on two real variables $s,t\in [a,b] \subset \R$ is called a {\em two-parameter flow of homeomorphisms of $\R^d$ on $[a,b]$} if the mapping $X_{s,t}$ is a homeomorphism on $\R^d$; $X_{s,s} = id$; $X_{s,t}^{-1} = X_{t,s}$ and $X_{s,t} = X_{u,t}\circ X_{s,u}$ for any $s,t,u\in [a,b]$. If in addition, $X_{s,t}$ is a linear operator for all $s,t \in [a,b]$, then the family $X_{s,t}$ is called a {\em two-parameter flow of linear operators of $\R^d$ on $[a,b]$}.

\begin{proposition}\label{thm.lin.flow}
	Suppose that the assumptions of  Proposition \ref{existence} are satisfied. Then the equation \eqref{lin1} generates a two-parameter flow of linear operators of $\R^d$ by means of its Cauchy operators.
\end{proposition}

\begin{proof}
	First note that the same method in the proof of Theorem \ref{existence} can be applied to prove the existence and uniqueness of solution $\Phi_\omega(t_0,t)$ of the  matrix-valued differential equation
	\begin{equation}\label{lin3}
	\Phi(t) = I + \int_{t_0}^t A(s)\Phi(s)ds + \int_{t_0}^t C(s)\Phi(s)d\omega(s),\ t \in [t_0,t_0+T].
	\end{equation}
	It is easy to show that the solution $\Phi_\omega(\cdot,\cdot):  \Delta^2 \to \R^{d\times d}$, with $\Delta^2 :=\{(s,t) \in [t_0,t_0+T] \times [t_0,t_0+T]: s\leq t\}$, has properties that  
	 $\Phi_\omega(s,s)= I_{d\times d}$ for all $s\geq 0$ and 
	\begin{equation}\label{stoTPF}
	\Phi_\omega(s,t)\circ \Phi_\omega(\tau,s) = \Phi_\omega(\tau,t),\quad \forall t_0\leq \tau\leq s \leq t \leq t_0+T. 
	\end{equation}
	The solution 
	$\Phi_\omega(\cdot,\cdot)$ is the mapping along trajectories of \eqref{linearfSDE} 
	 in forward time since YDE is directed. Like the ODE case, in our setting, the solution of the matrix equation \eqref{lin3} is the Cauchy operator of the vector equation \eqref{linearfSDE}.
	 
	Next, consider the adjoint matrix-valued pathwise differential equation  
	\begin{equation}\label{lin4}
	d\Psi(t_0,t) = -A^T(t) \Psi(t_0,t) dt -C^T(t) \Psi(t_0,t) d\omega(t) 
	\end{equation}
	with initial value $\Psi(t_0,t_0) =I$, and $A^T(\cdot), C^T(\cdot)$ are the transpose matrices of $A(\cdot)$ and $C(\cdot)$, respectively. By similar arguments we can prove that there exists a unique solution $\Psi_\omega(t_0,t)$ of \eqref{lin4}. Introduce the transformation $u(t) = \Psi_\omega(t_0,t)^T x(t)$. By the formula of integration by parts (see \cite[Proposition 6.12 and Exercise 6.13]{friz} or a fractional version in Z\"ahle \cite{zaehle}), we conclude that 
	\begin{eqnarray*}
		du(t) &=& [d\Psi_\omega(t_0,t)^T ] x(t) + \Psi_\omega(t_0,t)^T dx(t) \\
		&=&  [-\Psi_\omega(t_0,t)^T A(t) dt - \Psi_\omega(t_0,t)^T C(t) d\omega(t) ] x(t)+\Psi_\omega(t_0,t)^T [A(t)x(t) dt + C(t)x(t) d\omega(t)] \\
		&=& 0. 
	\end{eqnarray*}
	In other words, $u(t) = u(t_0)=x(t_0)=x_0$ or equivalently $\Psi_\omega(t_0,t)^T x(t) = x_0$. Combining with $\Phi$ in equation \eqref{lin3} we conclude that $\Psi_\omega(t_0,t)^T \Phi_\omega(t_0,t)x_0  = x_0$ for all $x_0 \in \R^d$, hence there exists $\Phi_\omega(t_0,t)^{-1}$ and $\Phi_\omega(t_0,t)^{-1} = \Psi_\omega(t_0,t)^T$. As a result, for any $x_0 \ne 0$ we have $\Phi_\omega(t_0,t)x_0 \ne 0 $ for all $t\geq t_0$. Thus we showed that the linear operator $\Phi_\omega(t_0,t)$, $t\geq t_0$, is nondegenerate. Similarly, for all  $t_0\leq s\leq t\leq t_0+T$ the operator $\Phi_\omega(s,t)$ is nondegenerate and $\Phi_\omega(s,t)^{-1} = \Psi_\omega(s,t)^T$. Putting $\Phi_\omega(t,s) := \Psi_\omega(s,t)^T$ for $t_0\leq s\leq t\leq t_0+T$ we have defined the family $\Phi_\omega(t,s)$ for all $s,t\in [t_0,t_0+T]$, and it is clearly a continuous two-parameter flow generated by \eqref{linearfSDE}.
\end{proof}
\begin{remark}
Using the solution formula for one dimensional system as in Section 3, one derives a Liouville - like formula as follow
$$\det\Phi_\omega(t_0,t) = \exp\left\{ \int_{t_0}^t {\rm trace\ } A(s)ds + \int_{t_0}^t {\rm trace\ } C(s)d\omega(s)\right\},$$
which also proves the invertibility of $\Phi_\omega(t_0,t)$.
\end{remark}
%%%%%%%%%%%%%%%%%%%%%%%%%%%%%%%%%%%%%%%%%%%%

\section{Lyapunov spectrum for nonautonomous linear system of YDEs}\label{Lyasec}
The classical Lyapunov spectrum of linear system of ordinary differential equations (henceforth abbreviated by ODEs) is a powerful tool in investigation of qualitative behavior of the system, see e.g. \cite{Bylov} or \cite{NeSt49}. Since \eqref{lin2} generates a two-parameter flow of homeomorphisms, we can instead study Lyapunov spectrum of the flow generated by the equation.
\subsection{Exponents and spectrum}
We aim to follow the technique in \cite{cong} and \cite{Mil86,Mil87}. From now on, let us consider the following assumptions on  the coefficients of \eqref{lin2}.\\

(${\textbf H}_1$)  $\hat{A}:=\|A\|_{\infty,\R^+} < \infty.$\\

(${\textbf H}_2$)  For some $\delta >0$, $\hat{C}:=\|C\|_{q{\rm-var},\delta,\R^+}:= \displaystyle\sup_{0\leq t-s \leq \delta} \|C\|_{q{\rm-var},[s,t]}< \infty$.\\
In (${\textbf H}_2$) we can assume, without loss of generality that $\delta = 1$. Put 
\begin{equation}\label{M0}
M_0:=\max\{\hat{A} , 2K\hat{C}\}
\end{equation}
where $K$ given by \eqref{constK}. It is obvious from \eqref{M*} that, for any $t_0\in\R^+$,
$$M^*(t_0,1)\leq M_0.$$
 
Note that conditions (${\textbf H}_1$), (${\textbf H}_2$) and Proposition \ref{existence} assure the existence and uniqueness of solution of \eqref{lin2} on $\R^+$. Moreover, Proposition \ref{thm.lin.flow} asserts that  \eqref{lin2} generates a two-parameter flow on $\R^d$ by means of its Cauchy operators $\Phi_\omega(\cdot,\cdot)$, and  $\Phi_\omega(s,t)x_0$ represents the value at time $t\in\R^+$ of the solution of \eqref{lin2} started at $x_0\in\R^d$ at time $s\in\R^+$. Following \cite{cong}, we introduce the notion of Lyapunov exponents of two-parameter flow of linear operators first, and then use it to define the Lyapunov exponents. We shall denote by ${\mathcal G}_k$ the Grassmannian manifold of all linear $k$-dimensional subspaces of $\R^d$.

Recall that for a real function $h: \R^+\rightarrow\R^d$ the {\em Lyapunov exponent of $h$} is the
number (which could be $\infty$ or $-\infty$)
$$
\chi(h(t)) := \limsup_{t\rightarrow\infty}\frac{1}{t}\log|h(t)|.
$$
(We make the convention that $\log$ is the logarithm of natural base and
$\log 0 := -\infty$.) 

\begin{definition}\label{dfn.LE2flow}
(i)\; Given a two-parameter flow $\Phi_\omega(s,t)$ of linear operators of $\R^d$ on the time interval $[t_0,\infty)$, the extended-real numbers (real numbers or symbol $\infty$ or $-\infty$)
\begin{equation}\label{eqn.LE.2flow}
\lambda_k(\omega) := \inf_{V\in {\mathcal G}_{d-k+1}} \sup_{y\in V} \limsup_{t\rightarrow \infty} \frac{1}{t} \log |\Phi_	\omega(t_0,t)y|,
\quad k=1,\ldots, d,
\end{equation}
are called  Lyapunov exponents of the flow $\Phi_\omega(s,t)$. The collection $\{ \lambda_1(\omega),\ldots, \lambda_d(\omega)\}$ is called  Lyapunov spectrum of the flow  $\Phi_\omega(s,t)$.\\
(ii)\; For any $u\in [t_0,\infty)$ the linear  subspaces of $\R^d$
\begin{equation}\label{eqn.flag.flow}
E_k^u(\omega) := \big\{ y\in\R^d\bigm| \limsup_{t\rightarrow\infty}
\frac{1}{t} \log |\Phi_\omega(u,t)y|  \leq \lambda_k(\omega) \big\},
\quad k=1,\ldots,d, \;\; 
\end{equation}
are called Lyapunov subspaces at time $u$ of the flow
$\Phi_\omega(s,t)$. The flag
of nonincreasing linear subspaces of $\R^d$
$$
\R^d = E_1^u(\omega) \supset E_2^u(\omega) \supset \cdots \supset
E_d^u(\omega) \supset \{0\}
$$
is called Lyapunov flag at time $u$ of the flow
$\Phi_\omega(s,t)$.\\
(iii)\; The  Lyapunov spectrum, Lyapunov exponents and Lyapunov subspaces of the linear YDE  \eqref{lin2} are those of the two-parameter flow $\Phi_\omega(s,t)$ generated by \eqref{lin2}.
\end{definition}
It is easily seen that the Lyapunov exponents in
Definition~\ref{dfn.LE2flow} are independent of $t_0$, and are ordered:
\[
\lambda_1(\omega) \geq \lambda_2(\omega) \geq \cdots \geq
\lambda_d(\omega), \qquad \omega\in\Omega.
\]
%==========================
Moreover, due to \cite[Theorems 2.5, 2.7, 2.8]{cong},  for any $u\in [t_0,\infty)$ and $k=1,\ldots, d$, the Lyapunov subspaces $E_k^u(\omega)$ are  invariant with respect to the flow in the following sense
\[
\Phi_\omega(s,t) E_k^s (\omega) = E_k^t(\omega),\qquad\hbox{for all}\;
s,t\in [t_0,\infty), k=1,\ldots,d.
\]
The classical definition of Lyapunov spectrum of a linear system of ODE is based on the normal basis of the solution of the system (see \cite{demidovich}). Millionshchikov \cite{Mil86} pointed out that these definitions are equivalent. In the following remark we restate some facts in \cite{demidovich}.

\begin{remark}\label{remLE}
(i) For every invertible matrix $B(\omega)$, the matrix $\Phi_\omega(t_0,t)B(\omega)$  satisfies $$\sum_{i=1}^d \alpha_i(\omega)\geq \sum_{i=1}^d\lambda_i(\omega)$$
	where $\alpha_i(\omega) $ is the Lyapunov exponent of its $i^{th}$ column.\\
(ii) Furthermore, we have Lyapunov inequality
	\[
	\sum_{i=1}^d\lambda_i \geq \displaystyle\limsup_{t\to \infty} \frac{1}{t}\log|\det\Phi_{\omega}(t_0,t)|.
	\]
Note that if Lyapunov exponents $\{\alpha_i(\omega), i=1,\dots,d\}$ of the columns of the matrix $\Phi_\omega(t_0,t)B(\omega)$ satisfy the equality $\sum_{i=1}^d\alpha(\omega) =\limsup\limits_{t\to\infty} \frac{1}{t}\log|\det\Phi_{\omega}(t_0,t)|$ then 
	  $\{\alpha_1(\omega),\dots, \alpha_d(\omega)\} $ is the spectrum of the flow $\Phi_{\omega}(s,t)$, i.e
	$$
	\{\alpha_i(\omega),i=1,\dots,d\} = \{\lambda_i(\omega),i=1,\dots,d\},
	$$
(but the inverse is not true).
\end{remark}
%=============================
%%%%%%%%%%%%%%%%%%%%%%%%%%%%
Now let us consider the following assumptions on the driving path $\omega$.\\

 (${\textbf H}_3$)  $\lim \limits_{n \to \infty\atop n\in\N} \frac{1}{n} \ltn \omega \rtn^p_{p-{\rm var},[n,n+1]} = 0.$\\
 
 (${\textbf H}_3^\prime$)  $\lim \limits_{n \to \infty\atop n\in\N} \frac{1}{n} \sum_{k =0}^{n-1}\ltn \omega \rtn^p_{p-{\rm var},[k,k+1]} = \Gamma_p(\omega) < \infty.$\\
 
It is easy to see that assumption (${\textbf H}_3^\prime$) implies (${\textbf H}_3$). We formulate below the first main result of this paper on the Lyapunov spectrum of equation \eqref{lin2}. 
\begin{comment}
Indeed,
 	\begin{eqnarray*}
 		&& \lim \limits_{n \to \infty} \frac{1}{n}  \ltn \omega \rtn^p_{p-{\rm var},[n,n+1]}\\
 		&=& \lim \limits_{n \to \infty} \frac{1}{n+1}  \ltn \omega \rtn^p_{p-{\rm var},[n,n+1]} \\
 		&=& \lim \limits_{n \to \infty} \Big( \frac{1}{n+1} \sum_{k=0}^{n} \ltn \omega \rtn^p_{p-{\rm var},[k,k+1]}-  \frac{1}{n} \sum_{k=0}^{n-1} \ltn \omega \rtn^p_{p-{\rm var},[k,k+1]} -  \frac{1}{n+1} \frac{1}{n}\sum_{k=0}^{n-1} \ltn \omega \rtn^p_{p-{\rm var},[k,k+1]}\Big)\\
 		&=& 0.
 	\end{eqnarray*}
 	
\end{comment}
 
%%%%%%%%%%%%%%%%%%%%%%%%%%%%

\begin{theorem}\label{thm2}
	Let $\Phi_{\omega}(s,t)$ be the two-parameter flow generated by \eqref{lin2} and $\{\lambda_1(\omega),\ldots,\lambda_d(\omega)\}$ be the Lyapunov spectrum of the flow $\Phi_{\omega}(s,t)$, hence of equation \eqref{lin2}. Then under assumptions (${\textbf H}_1$), (${\textbf H}_2$), (${\textbf H}_3$), %\eqref{omegaest}, 
	the Lyapunov exponents $\lambda_k(\omega), k=1,\ldots,d,$ can be computed via a discrete-time interpolation of the flow, i.e.
	\begin{equation}\label{est10} 
		\lambda_k(\omega) := \inf_{V \in \mathcal{G}_{d-k+1}} \sup_{y\in V} \limsup \limits_{\mathbb{N} \ni t\to \infty} \frac{1}{t} \log |\Phi_{\omega}(t_0,t)y|,\ k = 1,\ldots,d.
	\end{equation}
	In addition, if condition (${\textbf H}_3^\prime$) is satisfied, then 
	\begin{equation}\label{lyaestimate}
|\lambda_k(\omega)| \leq	\eta \Big[ 2 + \Big(\frac{2M_0}{\mu}\Big)^{p} (1+ \Gamma_p(\omega) )\Big], \quad \forall k = 1,\dots,d,	
	\end{equation}
	where $M_0$ is determined by \eqref{M0}, $0<\mu<\min\{1,M_0\}$ and $\eta= -\log(1-\mu)$.
\end{theorem}

\begin{proof}  
Recall from \eqref{growth} that for each $s\in \R^+$
\begin{eqnarray}\label{phiest-new}
\sup_{t\in [s,s+1]} \log \|\Phi_\omega(s,t)\| &\leq& \eta\Big[2+ (\frac{2M_0}{\mu})^p(1+\ltn \omega\rtn^p_{p{{\rm -var}},[s,s+1]})\Big].
\end{eqnarray}	

Fix $k\in\{1,\ldots,d\}$ and $y\in \R^d$.
Suppose  $0\leq t_0<t_1<t_2<t_3\cdots$ is an  increasing sequence of positive
real numbers on which the upper limit
$$
\limsup_{t\rightarrow\infty} \frac{1}{t} \log |\Phi_\omega(t_0,t)y| =:
z\in{\bar\R}
$$
is realized, i.e.,
$$
\lim_{m\rightarrow\infty} \frac{1}{t_m} \log |\Phi_\omega(t_0,t_m)y| = z.
$$
Let $n_m$ denotes the largest natural number which is smaller than or
equal to
$t_m$. Using the flow property of $\Phi_{\omega}(s,t)$ and 
assumption (${\textbf H}_3$) we have
\begin{eqnarray*}
z &=& \lim_{m\rightarrow\infty} \frac{1}{t_m} \log
|\Phi_\omega(t_0,t_m)y| \\
&=& \lim_{m\rightarrow\infty} \frac{1}{t_m} \log(
|\Phi_\omega(n_m,t_m)\Phi_\omega(t_0,n_m)y|) \\
&\leq& \lim_{m\rightarrow\infty} \frac{1}{t_m} \Big( \log
\|\Phi_\omega(n_m,t_m)\| + \log(|\Phi_\omega(t_0,n_m)y|)\Big) \\
&\leq& \limsup_{m\rightarrow\infty} \frac{1}{n_m} \log 
|\Phi_\omega(t_0,n_m)y| + \limsup_{m\rightarrow\infty} \frac{1}{n_m} \eta[2+ (\frac{2M_0}{\mu})^p(1+\ltn \omega\rtn^p_{p{{\rm -var}},[n_m,n_m+1]})]\\
&=&  \limsup_{m\rightarrow\infty}
\frac{1}{n_m}\log|\Phi_\omega(t_0,n_m)y|\\
&\leq&\limsup_{t\rightarrow\infty\atop t\in\N} \frac{1}{t} \log
|\Phi_\omega(t_0,t)y|.
\end{eqnarray*}
On the other hand, 
$$
\limsup_{t\rightarrow\infty\atop t\in\N} \frac{1}{t} \log
|\Phi_\omega(t_0,t)y| \leq
\limsup_{t\rightarrow\infty} \frac{1}{t} \log |\Phi_\omega(t_0,t)y| = z.
$$
Consequently, for all  $k\in\{1,\ldots,d\}$ and
$y\in \R^d$, we have the equality
$$
\limsup_{t\rightarrow\infty\atop t\in\N} \frac{1}{t} \log
|\Phi_\omega(t_0,t)y| =
\limsup_{t\rightarrow\infty} \frac{1}{t} \log |\Phi_\omega(t_0,t)y|,
$$
which proves \eqref{est10}.\\
Next, assume condition (${\textbf H}_3^\prime$) is satisfied. Then 
\begin{eqnarray}\label{integest}
\limsup_{n \to \infty} \frac{1}{n} \log
|\Phi_\omega(t_0,n)y|&\leq& \limsup_{n \to \infty} \frac{1}{n} \left( \log\|\Phi_\omega(t_0,\lceil t_0\rceil)\|  + \sum_{j=\lceil t_0\rceil}^{n-1}\log
\|\Phi_\omega(j,j+1)\| \right)\nonumber\\
&\leq& \limsup_{n \to \infty} \frac{1}{n} \sum_{j=0}^{n-1} \eta \Big[2+ \Big(\frac{2M_0}{\mu}\Big)^{p} \Big(1+ \ltn \omega \rtn_{p{\rm-var},[j,j+1]}^{p} \Big)\Big] \nonumber \\
&\leq& \eta \Big[ 2 + \Big(\frac{2M_0}{\mu}\Big)^{p} (1+ \Gamma_p(\omega)) \Big].
\end{eqnarray}
Since $\Phi_\omega(s,t) = (\Psi_\omega(s,t)^T)^{-1}$ where $\Psi$ is the solution matrix of the adjoint equation \eqref{lin4}, it follows that
\begin{eqnarray*}
	\limsup_{n \to \infty} \frac{1}{n} \log
	|\Phi_\omega(t_0,n)y|\geq \limsup_{n \to \infty} -\frac{1}{n} \log
	\|\Psi_\omega(t_0,n)\|= -\liminf_{n \to \infty} \frac{1}{n} \log
	\|\Psi_\omega(t_0,n)\|.
\end{eqnarray*}
Hence, either 
\[
0\leq \limsup_{n \to \infty} \frac{1}{n} \log
|\Phi_\omega(t_0,n)y| \leq \eta \Big[ 2 + \Big(\frac{2M_0}{\mu}\Big)^{p} (1+ \Gamma_p(\omega)) \Big]
\]
or 
\[
0\geq \limsup_{n \to \infty} \frac{1}{n} \log
|\Phi_\omega(t_0,n)y| \geq  -\liminf_{n \to \infty} \frac{1}{n} \log
\|\Psi_\omega(t_0,n)\|,
\]
which yields
\[
0\leq \liminf_{n \to \infty} \frac{1}{n} \log
\|\Psi_\omega(t_0,n)\| \leq \limsup_{n \to \infty} \frac{1}{n} \log
\|\Psi_\omega(t_0,n)\| \leq \eta \Big[ 2 + \Big(\frac{2M_0}{\mu}\Big)^{p} (1+ \Gamma_p(\omega)) \Big]
\]
where the last inequality can be proved similarly to the one in \eqref{integest}. Hence \eqref{lyaestimate} holds. 
\end{proof}

\begin{remark}\label{rem.dis.step}
The discretization scheme in Theorem \ref{thm2} can be formulated for any step size $h>0$.
\end{remark}

%%%%%%%%%%%%%%%%%%%%%%%%%%%%%%%%%%%%5
\subsection{Lyapunov spectrum of triangular systems}
It is well known in the theory of ODE that a linear triangular system can be solved successively and its Lyapunov spectrum is easily computed via its coefficients. In this subsection we present our similar result for linear triangular systems of YDE, under addition assumptions. Let us consider the system
\begin{eqnarray}\label{trianglesys}
dX(t) = A(t)X(t)dt+C(t)X(t)d\omega(t)
\end{eqnarray}
in which, $X= (x_1,x_2,...,x_d)$, $A = (a_{ij}(t)),C=(c_{ij}(t))$ are $d$ dimensional upper triangular matrices of coefficient functions satisfying conditions  (${\textbf H}_1$), (${\textbf H}_2$), the driving path $\omega $ satisfies (${\textbf H}_3$) and also the additional assumption\\

(${\textbf H}_4$)  $\displaystyle\lim\limits_{n\to \infty \atop n \in \N} \dfrac{\Big|\int_0^n c_{ii}(s)d\omega(s)\Big|}{t} = 0$  for any elements $c_{ii}(t),\; i=1,\dots, d$ in the diagonal of $C$.\\

As a motivation of our ideas, (${\textbf H}_4$) is satisfied for almost all realization $\omega$ of  a fractional Brownian motion  (see Lemma \ref{lemma2a} in Section~\ref{sec.appendix} for the proof and \cite{mishura} for details on fractional Brownian motions). Another situation satisying (${\textbf H}_4$)  is the case in which $\omega(t) = t^{\alpha}$ with $0<\alpha<1$ and $C(\cdot)$ is continuous and bounded.\\
To see how assumption (${\textbf H}_4$) is applied, we first consider equation  \eqref{trianglesys} in the  one dimensional case
\begin{equation}\label{1dimsystem}
dz(t) = a(t)z(t)dt + c(t)z(t)d\omega(t), \; z(0) = z_0.
\end{equation}
Thanks to the integration by part formula (see Z\" ahle~\cite[Theorem 3.1]{zaehle}), \eqref{1dimsystem} can be solved explicitly as 
	\begin{equation}\label{solution1}
	z(t) = z_0e^{\int_{0}^ta(s)ds+\int_{0}^tc(s)d\omega(s)}.
	\end{equation}
Moreover, we have the following lemma.
%==============================
\begin{lemma}\label{lemma13}
	The following estimates hold  for any nontrivial solution $z\not\equiv 0$ of \eqref{1dimsystem} 
	\\
	$(i)\;\chi(z(t)) = \overline{a}$,  
	\\
	$(ii)\;\chi(\ltn z\rtn_{q{\rm-var},[t,t+1]}) \leq \overline{a}$,\\
	where $\overline{a}:=\displaystyle \limsup_{n\to\infty\atop n\in\N}\frac{1}{n} \int_0^na(s)ds $.
\end{lemma}
\begin{proof}
$(i)$ The statement is evident under the assumption (${\textbf H}_4$). Namely,\\ 
	\begin{eqnarray*}
		\chi(z(t))
		&=&\limsup_{n\to \infty \atop n\in \N}\left( \frac{\log |z_0|}{n}+\frac{\int_0^na(s)ds}{n}+\frac{\int_0^nc(s)d\omega(s)}{n}\right)= \limsup_{n\to \infty \atop n\in \N}\frac{\int_0^na(s)ds}{n}=\overline{a}.
	\end{eqnarray*} 

	$(ii)$ Due to linearity it suffices to prove for $z_0=1$. Introduce the notations  $f(t)= \int_0^ta(s)ds,\;\; g(t) = \int_0^tc(s)d\omega(s)$, then  $z(t) = e^{f(t).g(t)}$.
	We have 
	$$
	\ltn f\rtn_{q{\rm-var},[s,t]}\leq (t-s)\|a\|_{\infty,[s,t]},\;  \ltn g\rtn_{q{\rm-var},[s,t]}\leq K \|c\|_{q{\rm-var},[s,t]}\ltn\omega\rtn_{p{\rm-var},[s,t]},\; \hbox{for all} \;\; 0\leq s<t;
	$$
	  and
	   $$
	   \chi(e^{f(t)})=\overline{a}, \;\; \chi(e^{g(t)}) = 0. 
	   $$
	For given $\varepsilon>0$, there exists $D_1$ such that
	$$
	e^{f(s)}< D_1e^{(\overline{a}+\varepsilon/3)s}, \;\; e^{g(s)}<D_1 e^{\varepsilon s /3},\; \forall s\geq 0.
	$$
	Hence, for any $t_0\geq 0$, the estimates
	$$
	\|e^{f}\|_{\infty,[t_0,t_0+1]}\leq D_2e^{(\overline{a}+\varepsilon/3)t_0};\;\;\|e^{g}\|_{\infty,[t_0,t_0+1]}\leq D_2 e^{\varepsilon t_0/3}
	$$
	hold for $D_2=\max\{D_1, D_1e^{\overline{a}+\varepsilon/3}\}$. On the other hand, by the mean value theorem and the continuity of $f$, for any $s,t\in [t_0,t_0+1]$,
	\begin{eqnarray*}
		|e^{f(t)}-e^{f(s)}|&=& e^{f(\xi)} |f(t)-f(s)|, \;\; \xi\in [s,t]\\
		&\leq &\|e^{f}\|_{\infty,[t_0,t_0+1]} |f(t)-f(s)|\leq \|e^{f}\|_{\infty,[t_0,t_0+1]} \ltn f\rtn_{q{\rm-var},[s,t]},
	\end{eqnarray*}
	which yields $$\ltn e^f\rtn_{q{\rm-var},[t_0,t_0+1]}\leq \|e^{f}\|_{\infty,[t_0,t_0+1]} \ltn f\rtn_{q{\rm-var},[t_0,t_0+1]}.$$
	Similarly, 
	\[\ltn e^g\rtn_{q{\rm-var},[t_0,t_0+1]}\leq \|e^{g}\|_{\infty,[t_0,t_0+1]} \ltn g\rtn_{q{\rm-var},[t_0,t_0+1]}.
	\]
	For $s,t\in [t_0,t_0+1]$,
	\begin{eqnarray*}
		|z(t)-z(s)|&=& |e^{f(t)g(t)}-e^{f(s)g(s)}|\\
		&\leq & e^{f(t)}|e^{g(t)}-e^{g(s)}|+ e^{g(s)}|e^{f(t)}-e^{f(s)}|\\
		&\leq & \|e^{f}\|_{\infty,[t_0,t_0+1]}\|e^{g}\|_{\infty,[t_0,t_0+1]} \left(\ltn f\rtn_{q{\rm-var},[s,t]}+\ltn g\rtn_{q{\rm-var},[s,t]}\right),
	\end{eqnarray*}
	hence by using Minkowski inequality we get
	\begin{eqnarray*}
		\ltn z\rtn_{q{\rm-var},[t_0,t_0+1]}&\leq & \|e^{f}\|_{\infty,[t_0,t_0+1]}\|e^{g}\|_{\infty,[t_0,t_0+1]} \left(\ltn f\rtn_{q{\rm-var},[t_0,t_0+1]}+\ltn g\rtn_{q{\rm-var},[t_0,t_0+1]}\right)\\
		&\leq& D_2^2e^{(\overline{a}+2\varepsilon/3)t_0} (\|a\|_{\infty,\R^+}+K\|c\|_{q{\rm-var},1,\R^+}\ltn\omega\rtn_{p{\rm-var},[t_0,t_0+1]}).
	\end{eqnarray*}
Note that condition (${\textbf H}_3$) implies  the boundedness of  $\frac{\ltn\omega\rtn_{p{\rm-var},[t_0,t_0+1]}}{t_0}$, $t_0\in \R^+$.  Therefore, there exists a constant  $D_3$ such that
$$
\ltn z\rtn_{q{\rm-var},[t_0,t_0+1] }\leq D_3 e^{(\overline{a}+\varepsilon)t_0},
$$
which proves (ii). 
\end{proof}
\medskip 

%===============================
Next we will show by induction that the Lyapunov spectrum of system \eqref{trianglesys} is $\{\overline{a}_{kk}, 1\leq k\leq d\}$ with $\overline{a}_{kk}:=\lim\limits_{t\to\infty}\frac{\int_0^ta_{kk}(s)ds}{t}$, provided that the limit is well-defined and exact.

The following lemma is a modified version of Demidovich~\cite[Theorem 1, p. 127]{demidovich}.
\begin{lemma}\label{lemma14}
  Assume that $g^i:\R^+\rightarrow  \R$, $i=1,\ldots, n$, are continuous functions of finite $q$-variation norm on any compact interval of $\R^+$, which satisfy
    $$
    \chi(g^i(t)),\;\chi(\ltn g^i\rtn_{q{\rm-var},[t,t+1]})\leq \lambda_i\in \R, \; i=1,\dots ,n.
    $$  
    Then 
    
    \noindent
$(i)\;$ $ \chi (\sum_{i=1}^n g^i(t)),\;\chi\left(\ltn \sum_{i=1}^n g^i\rtn_{q{\rm-var},[t,t+1]}\right) \leq \max_{1\leq i\leq n}\lambda_i,$\\
$(ii)\;$ $ \chi (\prod_{i=1}^n g^i(t)),\;\chi\left(\ltn \prod_{i=1}^n g^i\rtn_{q{\rm -var},[t,t+1]}\right) \leq \sum_{i=1}^n\lambda_i.$

\end{lemma}
\begin{proof}
$(i)\; $ The proof is similar to \cite[Theorem 1, p. 127]{demidovich} with note that
$$\ltn \sum_{i=1}^ng^i\rtn_{q{\rm-var},[t,t+1]}\leq \sum_{i=1}^n \ltn g^i\rtn_{q{\rm-var},[t,t+1]}.$$
\\
$(ii)\;$ The first inequality is known due to  \cite[Theorem 2, p\ 19]{demidovich}. For the second one, it suffices to show for $k=2$, since the general case is obtained by induction.\\ %\cite[Theorem1.5, p61]{demidovich}\\
It follows from Lemma \ref{qnorm-product} that
\begin{eqnarray*}
\ltn g^1g^2\rtn_{q{\rm-var},[t,t+1]}&\leq & (\|g^1\|_{\infty,[t,t+1]}+ \ltn g^1\rtn_{q{\rm-var},[t,t+1]})(\|g^2\|_{\infty,[t,t+1]}+ \ltn g^2\rtn_{q{\rm-var},[t,t+1]})\\
&\leq & 4 (\|g^1(t)\|+ \ltn g^1\rtn_{q{\rm-var},[t,t+1]})(\|g^2(t)\|+ \ltn g^2\rtn_{q{\rm-var},[t,t+1]}).
\end{eqnarray*}
Therefore the the second inequality followed from the first one and $(i)$.
\end{proof}

\vspace{0.5cm}
By similar arguments using the integration by part formula, the non-homogeneous one dimensional linear equation 
\begin{equation}
dx(t) = [a(t)x(t)+h_1(t)]dt + [c(t)x(t)+h_2(t)]d\omega(t)
\end{equation}
can be solved explicitly as
\begin{eqnarray*}
x(t) = e^{\int_0^ta(s)ds+\int_0^tc(s)d\omega(s)}\left(x_0 +  \int_0^t e^{-\int_0^ta(s)ds-\int_0^tc(s)d\omega(s)}h_1(s) ds + \int_0^t e^{-\int_0^ta(s)ds-\int_0^tc(s)d\omega(s)}h_2(s)d\omega(s)\right),\notag\\
\end{eqnarray*}
provided that $h_1,h_2$ are in $\cC^{q\rm{-var}}([0,t],\R)$ for all $t>0$. 
This allow us to solve triangular systems  by substitution as seen in the following theorem.

%====================================
\begin{theorem}\label{triangularcase} 
	Under assumptions (${\textbf H}_1$) -- (${\textbf H}_4$), if there exist the exact limits
\begin{equation}\label{exactmean}
\overline{a}_{kk}:=\lim_{t\to\infty} \frac{1}{t}\int_{0}^t a_{kk}(s)ds ,\;\; k= 1, \dots, d,
\end{equation}
	 then the spectrum of system \eqref{trianglesys} is given by
	$$\{\overline{a}_{11},\overline{a}_{22},\dots ,\overline{a}_{dd}\}.$$
\end{theorem}
\begin{proof}
For all $k=1,2,...,d$, put $Y_k(t)=e^{ \int_0^t a_{kk}(s)ds +\int_0^t c_{kk}(s)d\omega(s)} $. Then due to Lemma \ref{lemma13}
$$\chi(Y_k(t)) = \overline{a}_{kk},\; \chi (\ltn Y_k\rtn_{q{\rm -var},[t,t+1]})\leq \overline{a}_{kk} ,\; \chi(Y^{-1}_k(t)) =- \overline{a}_{kk},  \;\chi (\ltn Y^{-1}_k\rtn_{q{\rm -var},[t,t+1]})\leq -\overline{a}_{kk}.$$

 We construct a fundamental solution matrix $X(t)=\left(x_{ij}(t)\right)_{d\times d}$ of \eqref{trianglesys} as follows.
$$x_{ik}(t)=
\begin{cases}
0 & {\rm if}  \; i>k,\\
Y_k(t)& { \rm if}  \; i=k,\\
 Y_i(t) \left[\displaystyle\int_{t_{ik}}^t Y^{-1}_i(s)\sum_{j=i+1}^k a_{ij}(s)x_{jk}(s)ds  + \int_{t_{ik}}^t Y^{-1}_i(s)\sum_{j=i+1}^k c_{ij}(s)x_{jk}(s)d\omega(s)  \right]    &       {\rm if}  \; i<k,
\end{cases}
$$
in which, $t_{ik}=
	\begin{cases}
	0,\;\;{\rm if}\;\; \overline{a}_{kk}-\overline{a}_{ii}\geq 0\\
	+\infty ,\;\;{\rm if}\;\; \overline{a}_{kk}-\overline{a}_{ii}< 0.
	\end{cases}
	$\\
Now we consider the $d^{th}$ collumn of $X$ and prove by induction that 
$$\chi(x_{jd}(t)), \; \chi(\ltn x_{jd}\rtn_{q{{{\rm -var}}},[t,t+1]})\leq  \overline{a}_{dd},\; j=1,2,\dots, d.$$
First, by Lemma \ref{lemma13} the statement is true for $j=d$. Assume that $\chi(x_{jd}(t)), \; \chi(\ltn x_{jd}\rtn_{q{{{\rm -var}}},[t,t+1]})\leq  \overline{a}_{dd}$ for all $i+1\leq j\leq d$, we will prove that
$$\chi(x_{id}(t)), \; \chi(\ltn x_{id}\rtn_{q{{{\rm -var}}},[t,t+1]})\leq  \overline{a}_{dd}.$$
Put  
$$
I(t):=\displaystyle\int_{t_{id}}^t Y^{-1}_i(s)\sum_{j=i+1}^d a_{ij}(s)x_{jd}(s)ds \; \;{\rm{and }}\;\; J(t):=\displaystyle \int_{t_{id}}^t Y^{-1}_i(t)\sum_{j=i+1}^d c_{ij}(s)x_{jd}(s)d\omega(s),
 $$
 then 
$$
x_{id}(t) = Y_i(t) [I(t)+J(t)].
$$
Since $A$ is bounded, we apply \cite[ Corollary of Theorem 2, p. 129]{demidovich} to get 
\begin{eqnarray}
	\chi\left(\sum_{j=i+1}^{d} a_{ij}(s)x_{jd}(s)\right)\leq \overline{a}_{dd}. \label{a}
\end{eqnarray}
Therefore, $\chi\left(Y^{-1}_i(s) \sum_{j=i+1}^{d} a_{ij}(s)x_{jd}(s)\right)\leq \overline{a}_{dd}-\overline{a}_{ii}$.  Due to \cite[Theorem 4,p. 131]{demidovich} we obtain 
\[ \chi(I(t))\leq  \overline{a}_{dd}-\overline{a}_{ii}.\]
On the other hand, the following estimate holds
\[ \chi(\ltn I\rtn_{q{{\rm -var}},[t,t+1]})\leq  \overline{a}_{dd}-\overline{a}_{ii}.\]
 Indeed, with $I(t) = \int_0^t k(s)ds  $  and $\chi(k(s))\leq \lambda$, we have for $u,v\in[t,t+1]$,
\begin{eqnarray*}
|I(u)-I(v)|&\leq & |u-v| \|k\|_{\infty,[u,v]}\\
&\leq & |u-v| D(\varepsilon)e^{(\lambda+\varepsilon)t}, \text{\ for each\ } \varepsilon >0.
\end{eqnarray*}
This implies $\ltn I\rtn_{q{\rm-var},[t,t+1]}\leq D(\varepsilon)e^{(\lambda+\varepsilon)t}$.  The proof for the case $I(t) = \int_t^\infty k(s)ds  $ is similar.\\
Next, $\chi (Y^{-1}_i(t)),\chi(\ltn Y^{-1}_i\rtn_{q{{\rm -var}},[t,t+1]})\leq - \overline{a}_{ii}$ and $C$ satisfies (${\textbf H}_2$), i.e $\chi(C(t)),\;\chi(\ltn C\rtn_{q{\rm -var},[t,t+1]})\leq 0$. 
Together with the induction hypothesis that 
 $$\chi(x_{jd}(t)), \; \chi(\ltn x_{jd}\rtn_{q{{{\rm -var}}},[t,t+1]})\leq  \overline{a}_{dd}, \;\;\forall i+1\leq j\leq d$$
   and Lemma \ref{lemma14} we obtain
 $$\chi\left(Y^{-1}_i(t)\sum_{j=i+1}^d c_{ij}(t)x_{jd}(t)\right),\; \chi\left(\ltn Y^{-1}_i\sum_{j=i+1}^d c_{ij}x_{jd} \rtn_{q{\rm -var},[t,t+1]}\right)\leq \overline{a}_{dd}-\overline{a}_{ii}.$$
 Due to Lemma \ref{lemma11} and \ref{lemma12},
 \[
  \chi(J(t))\leq  \overline{a}_{dd}-\overline{a}_{ii}, \;  \chi(\ltn J\rtn_{q{\rm -var},[t,t+1]})\leq  \overline{a}_{dd}-\overline{a}_{ii}.
 \]
 Again, we apply Lemma \ref{lemma14} for $Y_i$, $I$ and $J$ to get
 $$
 \chi(x_{id}(t)), \; \chi(\ltn x_{id}\rtn_{q{{{\rm -var}}},[t,t+1]})\leq \overline{a}_{ii} +\overline{a}_{dd}-\overline{a}_{ii}=\overline{a}_{dd}.
 $$
 Hence, the Lyapunov exponent of the column  $d^{th}$, $X_d$, of matrix $X$ does not exceed $\overline{a}_{dd}$, meanwhile $\chi(x_{dd}(t)) = \overline{a}_{dd}$. This proves $\chi(X_d(t))=\overline{a}_{dd}$.
 
 Similarly, $\chi(X_i(t))=\overline{a}_{ii}$ for $i=1,2,\dots,d$, in which $X_i$ is the column $i^{th}$ of $X$.
 Finally, since $\sum_{i=1}^d\overline{a}_{ii} = \lim_{t\to\infty}\frac{1}{t}\log |\det X(t)|$, $X(t)$ is a normal matrix solution to \eqref{trianglesys} and the Lyapunov spectrum of \eqref{trianglesys} is
 $\{\overline{a}_{11},\overline{a}_{22},\dots ,\overline{a}_{dd}\}$.
 	\end{proof}
	
\begin{remark}
 In the theory of ODEs, Theorem Perron states that a linear equation can be reduced to a linear triangular system (see \cite[p. 180]{demidovich}). However, we do not know if it is true for linear Young differential equations. That is because for a linear YDE, besides the drift term $A$ corresponding to $dt$ we do have also the diffusion term $C$ corresponding to $d\omega$. Hence it is difficult to tranform the original system to a triangular form whose coefficient matrices only depend on $t$.
\end{remark}
%===========================
\subsection{Lyapunov regularity}
The concept {\em regularity} has been introduced by Lyapunov for linear ODEs, and since then has attracted lots of interests (see e.g. \cite[Chapter 3, p. 115]{arnold}, \cite{cong2004}, \cite{mil68}, or \cite[Section 1.2]{barreira}). For a linear YDE, we define the concept of Lyapunov regularity via the generated two-parameter flow. 

%===========================
\begin{definition}\label{dfn.Lregular}
	Let $\Phi_{\omega}(s,t)$ be a two-parameter flow of linear
	operators of $\R^d$ and
	$\{\lambda_1(\omega),\ldots, \lambda_d(\omega)\}$ be the Lyapunov spectrum
	of $\Phi_{\omega}(s,t)$.
	Then the non-negative $\bar\R$-valued random variable
	$$
	\sigma(\omega) :=  \sum_{k=1}^d \lambda_k
	-\liminf_{t\rightarrow\infty} \frac{1}{t} \log |\det\Phi_\omega(0,t)|
	$$ %\sum_{k=1}^d \lambda_k(\omega) -{\underline\delta}(\omega) =
	is called coefficient of nonregularity of  the two-parameter flow
	$\Phi_{\omega}(s,t)$.\smallskip\\
	The  coefficient of nonregularity of the linear YDE
	(\ref{lin2}) is, by definition, the coefficient of
	nonregularity of the two-parameter flow  generated by
	(\ref{lin2}).\smallskip\\
	A two-parameter flow is called Lyapunov regular if its coefficient
	of nonregularity equals 0 identically.
	A linear YDE is called Lyapunov regular if its coefficient of
	nonregularity equals 0.
\end{definition}
It follows from \cite{cong} that if a two-parameter linear flow $\Phi_{\omega}(s,t)$ is	Lyapunov regular then its determinant $\det\Phi_{\omega}(s,t)$ as well as any trajectory have exact Lyapunov exponents, i.e. the limit in \eqref{eqn.LE.2flow} is exact.\\
%=======================
We define the {\em adjoint equation} of \eqref{lin1} (and also of the equivalent integral equation \eqref{lin2}) by
\begin{equation}\label{adjoint}
dy(t) = -A^T(t)y(t) dt -C^T(t) y(t)d\omega(t). 
\end{equation}
The following lemma is a version of Perron Theorem from  the classical ODE case.
\begin{lemma}[Perron Theorem]\label{lem.perron}
	Let $\alpha_1\geq \cdots \geq \alpha_d$ and $\beta_1\leq \cdots \leq
	\beta_d$ be the Lyapunov spectrum of \eqref{lin2} and \eqref{adjoint} respectively. Then \eqref{lin2} is Lyapunov regular if and only if $\alpha_i + \beta_{i}=0$ for all $i=1,\ldots,d$. 
\end{lemma}

%=============================================
\begin{proof}
 The proof goes line by line with the ODE version in Demidovich ~\cite[p. 170-173]{demidovich}.
\end{proof}
%=======================
\begin{theorem}[Lyapunov theorem on regularity of triangular system]\label{thm.L.reg.triangle} 
	Suppose that the matrices $A(t), C(t)$ are upper triangular and satisfy (${\textbf H}_1$) -- (${\textbf H}_4$). 
	Then  system \eqref{trianglesys} is Lyapunov
	regular if and only if  there exists $\lim\limits_{t\to\infty} \frac{1}{t}\int_{t_0}^t a_{kk}(s)ds, \; k=\overline{1,d}$.
\end{theorem}
\begin{proof}
	The only if part is proved in Theorem \ref{triangularcase}. For the if part, the proof is similar to the \cite[p. 174]{demidovich}. Indeed, based on the normal basis of $\R^d$ which forms the unit matrix we construct a fundamental basis $ \tilde{X}$ of the system which is an upper triangular matrix and the diagonal entry is  
	$$Y_{1}(t),Y_{2}(t),\dots, Y_{d}(t),$$
	where $Y_{k}$ are defined in Theorem \ref{triangularcase}.\\
	We choose an upper triangular matrix $D=D(\omega)$ of which diagonal elements  are 1, such that $X:=\tilde{X}D$ is an normal basis of \eqref{lin2} with $x^i$ to be the column vectors (see also Remark \ref{remLE}). Put  $Y=(y_{ij})=(X^{-1})^T$ and repeat the arguments in Lemma \ref{lem.perron} under the regularity assumption, it follows that $Y$ is a normal basis of \eqref{adjoint}. Moreover, $ y_{kk}= Y_{k}^{-1}$ and 
	$$
	\chi(x^k(t))+\chi(y^k(t))=0,\forall k = 1,\dots,d.
	$$
	Hence
	\begin{eqnarray*}
		\chi(x^k(t))&\geq& \chi(Y_{k}(t)) = \limsup_{t\to\infty}\frac{1}{t}\int_{t_0}^ta_{kk}(s)ds
	\end{eqnarray*}
	and similarly,
	\begin{eqnarray*}
		\chi(y^k(t))&\geq& \chi(Y^{-1}_{k}(t)) = -\liminf_{t\to\infty}\frac{1}{t}\int_{t_0}^ta_{kk}(s)ds.
	\end{eqnarray*}
	Therefore, $$0\geq \limsup_{t\to\infty}\frac{1}{t}\int_{t_0}^ta_{kk}(s)ds-\liminf_{t\to\infty}\frac{1}{t}\int_{t_0}^ta_{kk}(s)ds\geq 0$$
	which implies that there exists the limit $\lim_{t\to\infty}\frac{1}{t}\int_{t_0}^ta_{kk}(s)ds, \;\; k = 1,\dots, d$.
\end{proof}

%====================================

\section{Lyapunov spectrum for linear stochastic differential equations}

In this section, we would like to investigate the same question in the random perspective, i.e. the driving path $\omega$ is a realization of a stochastic process $Z$ with stationary increments. System \eqref{lin1} can then be embedded into a stochastic differential equation, or precisely a random differential equation which can be solved in the pathwise sense. Such a system generates a stochastic two-parameter flow, hence it makes sense to study its Lyapunov spectrum and also to raise the question on the non-randomness of the spectrum.  

\subsection{Generation of stochastic two-parameter flows}

More precisely, recall that $\cC^{0,p-\rm{var}}([a,b],\R^d)$ is the closure of smooth paths from $[a,b]$ to $\R^d$ in $p$-variation norm and $\cC^{0,p-\rm{var}}(\R,\R^d)$ is the space of all $x: \R\to \R^d$ such that $x|_I \in \cC^{0,p-\rm{var}}(I, \R^d)$ for each compact interval $I\subset\R$. Then equip $\cC^{0,p-\rm{var}}(\R,\R^d)$ with the compact open topology given by the $p-$variation norm, i.e  the topology generated by the metric:
\[
d_p(x,y): = \sum_{m\geq 1} \frac{1}{2^m} (\|x-y\|_{p{\rm-var},[-m,m]}\wedge 1).
\]
Assign
$$\cC^{0,p-\rm{var}}_0(\R,\R^d):= \{x\in \cC^{0,p-\rm{var}}(\R,\R^d)|\; x(0)=0\}.$$
Note that for $x\in \cC^{0,p-\rm{var}}_0(\R,\R^d)$, $\ltn x\rtn_{p{\rm-var},I} $ and $\|x\|_{p{\rm-var},I}$ are equivalent norms for every compact interval $I$ containing $0$. \\
Let us consider a stochastic process $\bar{Z}$ defined on a complete probability space $(\bar{\Omega},\bar{\mathcal{F}},\bar{\bP})$  with realizations in $(\cC^{0,p-\rm{var}}_0(\R,\R), \mathcal{B})$, where $\mathcal{B}$ is Borel $\sigma -$algebra. Denote by $\theta$ the {\it Wiener shift}
\[
\theta_t m(\cdot) = m(t+\cdot) - m(t),\forall t\in \R, m\in \cC^{0,p-\rm{var}}_0(\R,\R^d).
\]
It is easy to check that $\theta$ forms a metric dynamical system $(\theta_t)_{t\in \R}$ on $(\cC^{0,p-\rm{var}}_0(\R,\R), \mathcal{B})$. Moreover, the Young integral satisfies the shift property with respect to $\theta$, i.e.
\begin{equation}\label{shift}
	\int_a^b x(u)d\omega(u) = \int_{a-r}^{b-r} x(u+r) d\theta_r \omega(u).  
\end{equation}  
Assume further that $\bar{Z}$ has stationary increments. It follows, as the simplest version for rough cocycle in \cite[Proposition 1]{bailleuletal} w.r.t. Young integrals that, there exists a probability $\bP$ on $(\Omega, \mathcal{F}) = (\cC^{0,p-\rm{var}}_0(\R,\R), \mathcal{B})$ that is invariant under $\theta$, and the so-called {\it diagonal process}  $Z: \R \times \Omega \to \R, Z(t,\tilde{\omega}) = \tilde{\omega}(t)$ for all $t\in \R, \tilde{\omega} \in \Omega$, such that $Z$ has the same law with $\bar{Z}$ and satisfies the {\it helix property}:
	$$Z_{t+s}(\omega) = Z_s(\omega) + Z_t(\theta_s\omega), \forall  \omega \in \Omega, t,s\in \R.$$
%The measurability of $Z$ comes from Castaing \& Valadier argument for separable space $\cC^{0,p-\rm{var}}_0(\R,\R)$. \\
Such stochastic process $Z$ has also stationary increments and almost all of its realization belongs to $\cC^{0,p-\rm{var}}_0(\R,\R)$. It is important to note that the existence of $\bar{Z}$ is necessary to construct the diagonal process $Z$. For example if $\bar{Z}$ is a fractional Brownian motion then the corresponding probability space  $(\bar{\Omega},\bar{\mathcal{F}},\bar{\bP})$ can be constructed explicitly as in \cite{mariaetal}. \\		

%===============================

Next, we consider the stochastic differential equation
\begin{equation}\label{stochlin1}
dx(t) = A(t)x(t) dt + C(t) x(t) d Z(t,\omega),\ x(t_0)=x_0 \in \R^d, t\geq t_0, 
\end{equation}
where the second differential is understood in the path-wise sense as Young differential. Under the assumptions in Proposition \ref{existence}, there exists, for almost sure all $\omega \in \Omega$, a unique solution to \eqref{stochlin1} in the pathwise sense with the initial value $x_0\in\R^d$. Moreover, the solution $X:[t_0,t_0+T]\times[t_0,t_0+T]\times\R^d \times\Omega \rightarrow \R^d$  satisfies: (i) for a.s. all $\omega \in \Omega$, $X(\cdot,a,x_0,\omega) \in \cC^{0,q{\rm-var}}([t_0,t_0+T],\R^d)$, and (ii) $X(t,\cdot,\cdot,\cdot)$   is measurable w.r.t $(a,x_0,\omega)$. As a result, the generated two parameter flow $\Phi_{\omega}(s,t): \R^d \rightarrow \R^d$ in Proposition \ref{thm.lin.flow} in the pathwise sense is also a {\it stochastic two-parameter flow} (see definition in \cite[p. 114]{kunita}). 

%=========================

\begin{proposition}\label{thm.LE.flow}
(i)The Lyapunov exponents $\lambda_k(\omega)$,
$k=1,\ldots, d$, of $\Phi_\omega(s,t)$ are measurable functions of
$\omega\in\Omega$. \\
(ii) For any $u\in [t_0,\infty)$, the Lyapunov subspaces $E_k^u(\omega)$,
$k=1,\ldots, d$, of $\Phi_\omega(s,t)$ are measurable with respect to
$\omega\in\Omega$, and  invariant with respect to
the flow in the following sense
\[
\Phi_\omega(s,t) E_k^s (\omega) = E_k^t(\omega),\qquad\hbox{for all}\;
s,t\in [t_0,\infty), \omega\in\Omega, k=1,\ldots,d.
\]
\end{proposition}
\begin{proof}
The proof of Theorem \ref{thm.LE.flow} is similar to the one in \cite[Theorems 2.5, 2.7, 2.8]{cong}.
\end{proof}

\begin{lemma}[Integrability condition]\label{lem3}
	Assume that 
	 there exists a function $H(\cdot,\cdot)$ which is increasing in the second variable, such that for any $r \geq 0$ 
	\begin{equation}\label{fBmest-new}
		E\ltn Z\rtn^r_{p{\rm-var},[s,t]} \leq H(r,t-s),\ \forall 0\leq s\leq t \leq 1.
	\end{equation}
Then under assumptions (${\textbf H}_1$) and (${\textbf H}_2$), $\Phi_\omega$ satisfies the following integrability condition for any $t_0\geq 0$
	\begin{equation}\label{integrability}
		E \sup_{t_0\leq s \leq t \leq t_0+1} \log^+ \|\Phi_\omega(s,t)^{\pm 1}\| \leq \eta \Big[2+ \Big(\frac{2M_0}{\mu}\Big)^{p} \Big(1+ H(p,1) \Big)\Big],
	\end{equation}
	where $M_0$ is determined by \eqref{M0}, $0<\mu<\min\{1,M_0\}$ and $\eta= -\log(1-\mu)$, and we use the notation  
	$$
	\log^+ \|\Phi_\omega(s,t)\| := \max \{ \log \|\Phi_\omega(s,t)\|, 0\}.
	$$
\end{lemma}
\begin{proof}
	The proof follows directly from \eqref{phiest-new} for $\Phi$ and $\Psi$, and from \eqref{fBmest-new}, 
	with note that for the inverse flow $\Psi_\omega(s,t)^{\rm T} = \Phi_\omega(s,t)^{-1}$ 
	\[
	\sup \limits_{t_0 \leq s \leq t \leq t_0+ 1} \log^+ \|\Phi_\omega(s,t)^{-1}\| = \sup\limits_{t_0 \leq s \leq t \leq t_0+1} \log^+ \|\Psi_\omega(s,t)\|
	\]
	and that \eqref{fBmest-new} is still satisfied for all $s,t\in [t_0,t_0+1]$ due to the increment stationary property of $Z$.
\end{proof}
%=========================

 Notice that condition \eqref{fBmest-new} derives (${\textbf H}_3^\prime$) for almost all driving paths $\omega$ due to Birkhorff ergodic theorem. Moreover, $\Gamma_p(\omega)$ is a random variable in $L^r(\Omega,\mathcal{F},\bP)$ for all $r>0$. If the metric dynamical system $(\Omega, \mathcal{F}, \bP, (\theta_t)_{t\in \R})$ is ergodic, it is known that $\Gamma_p(\omega)=E\ltn Z \rtn^p_{p{\rm-var},[0,1]} $ almost surely. As a result, the estimate \eqref{lyaestimate}  implies the following theorem.
\begin{theorem}\label{LEbounded}
Under assumptions (${\textbf H}_1$) and (${\textbf H}_2$) and condition \eqref{fBmest-new}, for each $k=1,\dots, d$ the Lyapunov exponent $\lambda_k(\omega)$ is of finite moments of any order $r>0$. More precisely,
\[
E|\lambda_k(\omega)|^r \leq \eta^r E \Big[ 2 + \Big(\frac{2M_0}{\mu}\Big)^{p} (1+ \Gamma_p(\omega)) \Big]^r, \quad \forall k = 1,\dots,d.
\]
In particular, if the metric dynamical system $(\Omega, \mathcal{F}, \bP, (\theta_t)_{t\in \R})$ is ergodic, the Lyapunov spectrum is bounded a.s. by non-random constants as follow, 
 \[
|\lambda_k(\omega)| \leq \eta \Big[ 2 + \Big(\frac{2M_0}{\mu}\Big)^{p} (1+E\ltn Z \rtn^p_{p{\rm-var},[0,1]} ) \Big], \quad \forall k = 1,\dots,d.
\]
\end{theorem}
%===========================
\begin{remark}\label{remomega}
	 Assumption \eqref{fBmest-new} is satisfied in case $Z$ is a fractional Brownian motion, see \cite[Corollary 1.9.2]{mishura} with $H> \frac{1}{2}$. Indeed, applying Garsia-Rademich-Rumsey inequality, see \cite[Lemma 7.3, Lemma 7.4]{nualart3} we see that for any fixed $r\geq 1$ and $\frac{1}{2}<\nu < H$ 
	\begin{eqnarray*}
	E\ltn Z\rtn^r_{p{\rm-var},[s,t]} \leq |t-s|^{\nu r} E \big(\ltn Z \rtn_{\nu{\rm-Hol},[s,t]}\big)^r
		\leq  C_{\nu,H,q,m} |t-s|^{\nu r} |t-s|^{(H-\nu)r} = C_{\nu,H,q,m} |t-s|^{H r} . 
	\end{eqnarray*}
	Moreover,  it is known in \cite{mariaetal} that $Z$ can be defined on a metric dynamical system $(\Omega, \mathcal{F}, \bP, (\theta_t)_{t\in \R})$  which is ergodic. 
\end{remark}

\subsection{Almost sure Lyapunov regularity}
In this subsection, for simplicity of presentation we consider all the equations on the whole time line $\R$. The half-line case $\R^+$ can be easily treated in a similar manner.
 
We start the subsection with a very special situation in which the coefficient functions are autonomous, i.e. $A(\cdot) \equiv A, C(\cdot) \equiv C$. In this case, the stochastic two-parameter flow $\Phi_\omega(s,t)$ of \eqref{stochlin1} generates a linear random dynamical system $\Phi^\prime$ (see e.g. Arnold~\cite[Chapter 1]{arnold} for the definition of  random dynamical systems). Indeed, from \eqref{shift} and the fact that
\begin{eqnarray*}
	x(t) &=& x_0 + \int_0^s A x(u)du + \int_0^s C x(u)d\omega(u) + \int_s^t A x(u)du + \int_s^t C x(u)d\omega(u)\\
	&=& x(s) +  \int_0^{t-s} A x(u+s)du + \int_0^{t-s} C x(u+s)d\theta_s\omega(u),
\end{eqnarray*}
it follows due to the autonomy that $\Phi_\omega(s,t)= \Phi(t-s,\theta_s \omega)$. Hence $\Phi^\prime(t,\omega) := \Phi_\omega(0,t)$ satisfies the {\it cocycle property}
\[
\Phi^\prime(t+s,\omega) = \Phi^\prime(t,\theta_s\omega) \circ \Phi^\prime(s,\omega).
\]
Assign $t_0 = 0$, if follows from \eqref{integrability} that
\[
\sup_{t\in [0,1]} \log \|\Phi^\prime(t,\omega)^{\pm 1}\|  \in L^1(\Omega,\mathbb{P}).
\] 

By applying the multiplicative ergodic theorem (see Oseledets~\cite{Ose68} and Arnold~\cite[Chapter 3]{arnold}) for $\Phi^\prime$ generated from \eqref{stochlin1}, there exists a Lyapunov spectrum consisting of exact Lyapunov exponents provided by the multiplicative ergodic theorem and it coincides with the Lyapunov spectrum defined in Definition \ref{dfn.LE2flow}. In addition, the flag of Oseledets' subspaces coincides with the flag of Lyapunov spaces.

The same conclusions hold if the system is periodic with period $r$, i.e. $A(\cdot+ r) = A(\cdot), C(\cdot+r) = C(\cdot)$. In fact, we can prove that $\Phi_\omega(s,t) = \Phi_{\theta_r \omega}(s-r,t-r)$ and
\[
\Phi(nr,\omega) = \Phi(r,\theta_{(n-1)r}\omega) \circ \Phi((n-1)r,\omega),\ \forall n \in \Z.
\]
In this case, \eqref{stochlin1} generates a discrete random dynamical system $\{\Phi(nr,\omega)\}_{n\in \Z}$	which satisfies the integrability condition \eqref{integrability}. Hence all the conclusions of the MET hold and almost all the path-wise system is Lyapunov regular.

In general, it might not be true that system \eqref{stochlin1} is regular for almost sure $\omega$. However, under the further assumptions of $A, C$, we can construct a linear random dynamical system such that almost sure all the pathwise systems are Lyapunov regular. The construction uses the so-called {\it Bebutov flow}, as investigated by Millionshchikov \cite{mil68}, \cite{Mil68} (see also \cite{JPS87}, \cite{sacker-sell}, \cite{sell}).
Specifically, assume that $A$ satisfies a stronger condition that
$$
({\textbf H}_1^\prime) \quad \hat{A}:=\|A\|_{\infty,[0,\infty)} < \infty \quad\hbox{and}\quad \lim \limits_{\delta \to 0} 
\sup \limits_{|t-s| < \delta} |A(t)-A(s)| =0.\hspace*{4cm}
$$
Consider the shift dynamical system $S^A_t(A)(\cdot) := A(\cdot +t)$ in the space $\mathcal{C}^b = \mathcal{C}^b(\R,\R^{d \times d})$ of bounded and uniformly continuous matrix-valued continuous function on $\R$ with the supremum norm. The closed hull $\mathcal{H}^A:=\overline{\cup_t S_t(A)}$ in $\mathcal{C}^b$ is then compact, hence we can construct on $\mathcal{H}^A$ a probability structure such that $(\mathcal{H}^A,\cF^A,\mu^A, S^A)$ is a probability space where $\mu^A$ is a $S$-invariant probability measure, see e.g. \cite[Theorem 4.9, p. 63]{karatzas}. 

When applying Millionshchikov's approach of using Bebutov flows to our system \eqref{stochlin1}, we need to construct not only $(\mathcal{H}^A,\cF^A,\mu^A, S^A)$, but also $(\mathcal{H}^C,\cF^C,\mu^C, S^C)$, with a little more regularity condition for $C$. Recall that $\cC^{0,\alpha-\rm{Hol}}([a,b],\R^{d \times d})$ is the closure of smooth paths from $[a,b]$ to $\R^{d \times d}$ in $\alpha$-H\"older norm and $\cC^{0,\alpha-\rm{Hol}}(\R,\R^{d \times d})$ is the space of all $x: \R\to \R^{d \times d}$ such that $x|_I \in \cC^{0,\alpha-\rm{Hol}}(I,\R^{d \times d})$ for each compact interval $I\subset\R$, equipped  with the compact open topology given by the H\"older norm, i.e  the topology generated by metric
\[
d(x,y): = \sum_{m\geq 1} \frac{1}{2^m} (\|x-y\|_{\alpha,[-m,m]}\wedge 1).
\]
Following \cite[Chapter 2, p. 62]{karatzas}, for any $c \in \cC^{0,\alpha-\rm{Hol}}(\R,\R^{d \times d})$, 
any interval $[a,b]$ and $\delta >0$, we define the {\it module of $\alpha$-H\"older} on $[a,b]$:
\[
m^{[a,b]}(c,\delta) := \ltn c \rtn_{\alpha,\delta,[a,b]} = \sup_{a\leq s<t \leq b, t-s \leq \delta} \frac{|c(t)-c(s)|}{|t-s|^\alpha}. 
\]
By the same arguments as in \cite[Theorem 4.9, Theorem 4.10, p. 62-64]{karatzas} we get the following result, of which the proof is given in the Appendix. 
\begin{lemma}\label{compact}
	A set $\cH \subset \cC^{0,\alpha-\rm{Hol}}(\R,\R^k)$ has a compact closure if and only if the following conditions hold:
	\begin{eqnarray}
	\sup_{c \in \cH} |c(0)| &<& \infty, \label{mcompact1}\\
	\lim \limits_{\delta \to 0} \sup_{c \in \cH} m^{[a,b]}(c,\delta) &=&  0,\quad \forall [a,b]. \label{mcompact2}
	\end{eqnarray}
\end{lemma}
To construct a Bebutov flow for $C$, assume that there exists $\alpha > \frac{1}{q}$ such that $C \in \cC^{0,\alpha-\rm{Hol}}(\R, \R^{d\times d})$ satisfies a condition  stronger  than (${\textbf H}_2$):
\begin{equation}\label{bounded}
{\rm ({\textbf H}_2^\prime)}\quad \|C\|_{\infty,\R} = \sup_{ t\in \R} |C(t)|< \infty \qquad \text{and} \qquad \lim \limits_{\delta \to 0} \sup_{-\infty <s< t < \infty, |t-s|\leq \delta} \frac{|C(t) - C(s) |}{|t-s|^\alpha} = 0.
\end{equation}
Consider the set of translations $C_r (\cdot) := C(r + \cdot) \in \cC^{0,\alpha-\rm{Hol}}(\R, \R^{d\times d})$. Under conditions \eqref{bounded}, Lemma \ref{compact} concludes that the closure set $\cH^C:=\overline{\{C_r: r\in \R\}}$ is compact on the separable complete metric space $\cC^{0,\alpha-\rm{Hol}}(\R,\R^{d \times d})$, in fact $\theta_t$ also preserves the norm on $\cC^{0,\alpha-\rm{Hol}}(\R,\R^{d \times d})$. The shift dynamical system $S^C_t c(\cdot) = c(t + \cdot)$ maps $\cH^C$ into itself, hence by Krylov-Bogoliubov theorem \cite[Chapter VI, \S9]{NeSt49}, there exists at least one probability measure $\mu^C$ on $\cH^C$ that is invariant under $S^C$, i.e. $\mu^C (S^C_t \cdot) = \mu^C(\cdot)$, for all $t \in \R$. 

It makes sense then to construct the product probability space $\mathbb{B} = \mathcal{H}^A \times \mathcal{H}^C \times \Omega$ with the product sigma field $\cF^A \times \cF^C \times \cF$, the product measure $\mu^{\mathbb{B}}:=\mu^A \times \mu^C \times \bP$ and the product dynamical system  $\Theta = S^A \times S^C \times \theta$ given by 
\[
\Theta_t(\tilde{A},\tilde{C},\omega) := (S^A_t(\tilde{A}), S^C_t(\tilde{C}), \theta_t \omega). 
\]     
Now for each point $b = (\tilde{A},\tilde{C},\omega) \in \mathbb{B}$, the fundamental (matrix) solution $\Phi^*(t,b)$ of the equation
\begin{equation}\label{eqn.nonauto.YDE}
dx(t) = \tilde{A}(t)x(t)dt + \tilde{C}(t)x(t)d\omega(t), \quad x(0)=x_0\in R^d,
\end{equation}
defined by $\Phi^*(t,b)x_0 := x(t)$, satisfies the cocycle property due to the existence and uniqueness theorem and the fact that
\begin{eqnarray*}
	x(t+s) &=& x_0 + \int_0^s \tilde{A}(u) x(u)du + \int_0^s \tilde{C}(u) x(u)d\omega(u) \\
	&&+ \int_s^{t+s} \tilde{A}(u) x(u)du + \int_s^{t+s} \tilde{C}(u) x(u)d\omega(u)\\
	&=& x(s) +  \int_0^{t} S^A_s(\tilde{A})(u) x(u+s)du + \int_0^{t} S^C_s(\tilde{C})(u) x(u+s)d\theta_s\omega(u).
\end{eqnarray*}
Therefore the nonautonomous linear YDE \eqref{eqn.nonauto.YDE} generates a cocycle (random dynamical system) $\Phi^*: \R \times \mathbb{B} \times \R^d \to \R^d$ over the metric dynamical system $(\mathbb{B},\mu^{\mathbb{B}})$. 
Thus, starting from investigation of one linear stochastic nonautonomous YDE \eqref{stochlin1} we consider its $\omega$-wise and embed to a Bebutov flow using Millionshchikov's approach \cite{Mil68}, henceforth construct a random dynamical system over the product probability space for which the following statement holds.
\begin{theorem}[Millionshchikov theorem] \label{thm.Millionshchikov}
	Under assumptions (${\textbf H}_1^\prime$), (${\textbf H}_2^\prime$) and \eqref{fBmest-new}, the nonautonomous linear stochastic ($\omega$-wise) Young equation \eqref{eqn.nonauto.YDE} is Lyapunov regular for almost all $b \in \mathbb{B}$ in the sense of the probability measure $\mu^{\mathbb{B}}$.
\end{theorem}
\begin{proof}
The integrability condition for the product probability measure $\mu^\mathbb{B}$ is a direct consequence of \eqref{integrability}. Hence all the conclusions of the multiplicative ergodic theorem hold for almost all $b \in \mathbb{B}$, which implies the Lyapunov regularity of \eqref{eqn.nonauto.YDE} for almost all $b\in \mathbb{B}$ in the sense of the probability measure $\mu^{\mathbb{B}}$.
\end{proof}
\begin{remark}
	(i) In \cite{mil68} and \cite{Mil68}, Millionshchikov proved the Lyapunov regularity (almost surely with respect to an arbitrary invariant measure of the Bebutov flow on $\cH^A$ generated by the ordinary differential equation $\dot{x} = A(t)x$), using the triangularization scheme provided by the Perron theorem for ordinary differential equations. In other words,  Millionshchikov obtained an alternative proof of the multiplicative ergodic theorem (see also Arnold~\cite[p.~112]{arnold}, Johnson, Palmer and Sell~\cite{JPS87}). In fact, Millionshchikov proved a bit stronger property than Lyapunov regularity that, almost all such systems are statistically regular. 
	
	(ii) Theorem~\ref{thm.Millionshchikov} can be viewed as a version of multiplicative ergodic theorem for a nonautonomous linear stochastic Young differential equation which uses combination of Millionshchikov \cite{Mil68} approach (topological setting using Bebutov flow for differential equation) and Oseledets \cite{Ose68} approach (measurable setting with probability space $(\Omega,\mathcal{F},\bP )$).
	
	(iii) It is important to note that, although for almost all $b\in \mathbb{B}$ the  nonautonomous linear stochastic ($\omega$-wise) Young equation \eqref{eqn.nonauto.YDE} is Lyapunov regular, it does not follow that the original system \eqref{stochlin1} is Lyapunov regular.
\end{remark}

\subsection*{Discussions on the non-randomness of Lyapunov exponents}

Since we are dealing with stochastic equation YDE \eqref{stochlin1} it is important and interesting to know whether its Lyapunov spectrum is nonrandom. We give here a brief discussion on this problem.

We remind the readers of the non-randomness of Lyapunov exponents $\lambda_1(\omega),\ldots, \lambda_d(\omega)$ for systems driven by standard Brownian noises (see e.g. \cite{cong, congquynh09}). Since Theorem~\ref{thm2}  still holds in that situation and the definition of Lyapunov exponent does not depend on the initial time $t_0$, it follows that $\lambda_k(\omega)$ is measurable with respect to the sigma algebra generate by $\{W(n+1)-W(n): n\geq m\}$ for any $m\geq 0$, thus measurable w.r.t. the tail sigma field $\cap_m \sigma (\{W(n+1)-W(n): n\geq m\})$. Due to pairwise independence of all variables of the form $W(n+1)-W(n)$, one can apply Kolmogorov's zero-one law \cite{karatzas} to conclude that Lyapunov exponents are in fact non-random constants. Thus we have nonrandomness of the Lyapunov spectrum in the case of nonautonomous linear stochastic differential equations driven by standard Brownian motions. Note that here the Lyapunov exponents of the systems can be nonexact.

In general, a stochastic process $Z$ does not have independent increments, thus it is difficult to construct such a filtration and to apply the Kolmogorov's zero-one law. However, the second case of nonrandom Lyapunov spectrum is the case of autonomous or periodic linear stochastic Young equations discussed at the beginning of this subsection where we may apply the classical Oseledets MET by exploiting autonomy or periodicity of the system. Note that in this case the probability measure is the probability measure of the process $Z$ and the Lyapunov exponents of the systems are exact.

The third case is triangular nonautonomous linear stochastic Young differential equations treated in Section \ref{Lyasec}. In this case, due to the triangular form of the system we may solve it successively and use explicit formula of the solution to derive  Theorem~\ref{triangularcase} showing that the Lyapunov spectrum consists of exact Lyapunov exponents and is nonrandom. Note that in this case the system in nonautonomous, the measure is the probability measure  of the process $Z$ and the Lyapunov exponents of the systems are exact.

For a general system \eqref{stochlin1} which satisfies assumptions (${\textbf H}_1^\prime$), (${\textbf H}_2^\prime$) and \eqref{fBmest-new}, the statement on the non-randomness of Lyapunov spectrum depends on whether the product dynamical system $\Theta$ is ergodic on the product probability measure $\mu^\mathbb{B}$, as a consequence of the Birkhorff ergodic theorem. The answer is then affirmative in case $S^A$ and $S^C$ are weakly mixing and $\theta$ is ergodic, i.e. $S^A$ (respectively $S^C$) satisfies the condition 
\[
\lim \limits_{n \to \infty} \frac{1}{n} \sum_{j=0}^{n-1} \Big|\mu^A\Big(S^A_j(Q_1) \cap Q_2\Big) - \mu^A (Q_1) \mu^A(Q_2)\Big| =0,\quad \forall Q_1\ne Q_2 \in \mathcal{F}^A
\]
(respectively for $S^C$). It is well known (see e.g. Ma\~n\'e~\cite[p. 147]{mane}) that the weak mixing of $S^A$ and $S^C$ implies the weak mixing of the product dynamical system $S^A \times S^C$ which, together with the ergodicity of $\theta$, implies the ergodicity of the product flow $\Theta$. The problem on non-randomness of Lyapunov spectrum can therefore be translated into the question on the weak-mixing of dynamical systems $S^A$ and $S^C$.  	

%================================================

\section*{Acknowledgments}

This research is partly funded by Vietnam National Foundation for Science and Technology Development (NAFOSTED) under grant number FWO.101.2017.01.

%===================================================

\section{Appendix}\label{sec.appendix}
\begin{proof}[{\bf Proposition \ref{existence}}]
	The proof follows the same techniques in \cite{nualart3} and in \cite{congduchong17} with some modifications. First, consider $x\in \cC^{q{\rm-var}}([a,b],\R^d)$ with some $[a,b] \subset[t_0,t_0+T]$. Define the mapping given by
	\begin{eqnarray}
		F(x)(t) &= &x(a) + I(x)(t)+J(x)(t)\nonumber\\
		&:=& x(a) + \int_{a}^t A(s)x(s) ds +\int_{a}^t C(s)x(s)d\omega(s), \quad\forall t\in [a,b].\label{eqn.F}
	\end{eqnarray}
	Then $F(x)\in   \cC^{p{\rm-var}}([a,b],\R^{d})$ and direct computations show that for every $[s,t]\subset [a,b]$
	\begin{eqnarray}
		\ltn Fx\rtn_{p{\rm-var},[s,t]}&\leq& P \| x\|_{q{\rm-var},[s,t]} \label{(i)}\\
		\ltn Fx-Fy\rtn_{p{\rm-var},[s,t]}&\leq&   P \| x-y\|_{q{\rm-var},[s,t]}, \label{(ii)}
	\end{eqnarray}
	where 
	\[
	P:=\|A\|_{\infty,[t_0,t_0+T]}(t-s)+2K\|C\|_{q{\rm-var},[t_0,t_0+T]}\ltn\omega\rtn_{p{\rm-var},[s,t]}, 
	\quad  K \; \hbox{is defined in \eqref{constK}}.
	\]
	Similar to \cite{congduchong17}, for a given $0<\mu<\min\{1,M^*\}$, where $M^*$ is defined by \eqref{M*}, we construct the sequence of strictly increase greedy times $\tau_n$ with $\tau_0 = 0$ satisfying
	\begin{equation}\label{stoptime1}
		(\tau_{k}-\tau_{k-1})+\ltn \omega\rtn_{p{\rm-var},[\tau_{k-1},\tau_k]}=\mu/M^*.
	\end{equation}
	Then $\tau_k\to \infty$ as $k\to \infty$ (see the proof in \cite{congduchong17}). Denote by $N(a,b,\omega)$ the number of $\tau_k$ in the finite interval $(a,b]$, then from \cite{congduchong17}  
	\begin{eqnarray}\label{N}
		N(t_0,t_0+T,\omega)-1&\leq&  \left(\frac{2M^*}{\mu}\right)^p(T^p+\ltn \omega\rtn^p_{p{\rm-var},[t_0,t_0+T]}).
	\end{eqnarray}
	Without loss of generality assume that $t_0=0$. Define the set 
	\[
	B=\{x\in \cC^{q{\rm-var}}([\tau_0,\tau_1],\R^d)|\ x(\tau_0)= x_{0},\;\|x\|_{q{\rm-var},[\tau_0,\tau_1]}\leq \frac{1}{1-\mu}|x_{0}|\}.
	\]
	It is easy to check that $B$ is a closed ball in Banach space $ \cC^{q{\rm-var}}([\tau_0,\tau_1],\R^d)$. Using \eqref{(i)}, \eqref{stoptime1} and the fact that $p<q$, we have 
	\begin{eqnarray*}
		\| F(x)\|_{q{\rm-var},[\tau_0,\tau_1]}&\leq& |x_{0}| + M^* (\tau_1-\tau_0 + \ltn \omega\rtn_{p{\rm-var},[\tau_0,\tau_1]} ) \|x\|_{q{\rm-var},[\tau_0,\tau_1]}\\
		&\leq& |x_{0}| +\mu \|x\|_{q{\rm-var},[\tau_0,\tau_1]}\\
		&\leq& \frac{1}{1-\mu}|x_{0}|, \ \text{for each\ } x\in B. 
	\end{eqnarray*}
	Hence, $F:B\rightarrow B$. On the other hand, by \eqref{(ii)} and \eqref{stoptime1}, for any $x,y\in B$
	$$
	\| F(x)-F(y)\|_{q{\rm-var},[\tau_0,\tau_1]} \leq \mu \| x-y\|_{q{\rm-var},[\tau_0,\tau_1]}.
	$$
	Since $\mu<1$, $F$ is a contraction mapping on $B$. We conclude that there exists a unique fixed point of $F$ in $B$ or there exists local solution of \eqref{linearfSDE} on $[\tau_0,\tau_1]$. By induction we obtain the solution on $[\tau_i,\tau_{i+1}]$ for $i=1,..,N(0,T,\omega)-1$ and finally on $[\tau_{N(0,T,\omega)-1},T]$. The global solution of \eqref{linearfSDE} then exists uniquely. From \eqref{(i)} it is obvious that the solution is in  $ \cC^{p{\rm-var}}([t_0,t_0+T],\R^d)$.  
	Estimate \eqref{growth} can then be derived using similar arguments to \cite[Remark 3.4iii]{congduchong17}. In fact,
	\[
	\|x\|_{\infty,[\tau_0,\tau_1]}\leq \|x\|_{q{\rm-var},[\tau_0,\tau_1]} \leq\frac{1}{1-\mu}|x_{0}|,
	\]
	thus by induction, it is evident that 
	\[
	\|x\|_{\infty,[t_0,t_0+T]}\leq |x_{0}| \left(\frac{1}{1-\mu}\right)^{N(t_0,t_0+T,\omega) +1}	\leq |x_{0}| e^{\eta [N(t_0,t_0+T,\omega) +1]}.
	\]
	Combining to \eqref{N}, we obtain \eqref{growth}.\\
	Since 
	\[
	\ltn x\rtn_{q{\rm-var},[s,t]}=\ltn F(x)\rtn_{q{\rm-var},[s,t]} \leq M^* (t-s + \ltn \omega\rtn_{p{\rm-var},[s,t]})\| x\|_{q{\rm-var},[s,t]}
	\]
	for all $s<t\in [t_0, t_0+T]$, by proving similarly to \cite[Corollary 3.5]{congduchong17} we obtain
	\begin{eqnarray*}
		\ltn x\rtn_{q{\rm-var},[t_0,t_0+T]}&\leq &|x_{0}| e^{(1+\eta)[3+ (\frac{2M^*}{\mu})^p(T^p+\ltn \omega\rtn^p_{p{\rm-var},[t_0,t_0+T]})]}.%
	\end{eqnarray*}

%{\bf .}\\
%\begin{comment}
	Next, fix $(x_0,\omega)$ and consider $(x_0',\omega')$ in the ball centered at $(x_0,\omega)$ with radius 1.\\
	Put $x(\cdot) = x(\cdot, t_0,x_0,\omega)$, $x'(\cdot) = x(\cdot, t_0,x'_0,\omega')$ and $y(\cdot) =x(\cdot)-x'(\cdot) $. By  \eqref{growth} and \eqref{growth1}, we can choose a positive number $D_1$ (depending on $M^*,x_0,\omega$) such that 
	$$\|x\|_{p{\rm-var},[t_0,t_0+T]}, \|x'\|_{p{\rm-var},[t_0,t_0+T]}\leq D_1 .$$
%	Firstly, without losing generality assuming that $a'\geq a$, we then have
%	\begin{eqnarray*}
%		|y(a')| &= & |x(a')-x_0'|\\
%		&\leq & |x_0 -x_0'| + \left|\int_a^{a'}A(s)x(s)ds\right| +\left|\int_a^{a'} C(s)x(s)d\omega(s)\right|\\
%		&\leq & |x_0 -x_0'| + D_1M^*|a-a'|  +D_1M^*\ltn \omega\rtn_{p{\rm-var},[a,a']}.
%	\end{eqnarray*}
	 We have
	\begin{eqnarray*}
		|y(t)-y(s)|&\leq & \Big|\int_s^t A(u)y(u)du\Big| + \Big|\int_s^t C(u)y(u)d\omega(u)\Big| + \Big|\int_s^t C(u)x'(u)d(\omega(u)-\omega'(u))\Big|\\
		&\leq& \|A\|_{\infty,[t_0,t_0+T]}\|y\|_{\infty,[s,t]}(t-s) + 2K\|C\|_{q{\rm-var},[t_0,t_0+T]}\|y\|_{p{\rm-var},[s,t]}\ltn\omega\rtn_{p{\rm-var},[s,t]}\\
		&& + 2K\|C\|_{q{\rm-var},[t_0,t_0+T]}\|x'\|_{p{\rm-var},[s,t]}\ltn\omega-\omega'\rtn_{p{\rm-var},[s,t]} \\
		&\leq &  M^* (t-s+\ltn\omega\rtn_{p{\rm-var},[s,t]})\|y\|_{p{\rm-var},[s,t]}+M^*\|x'\|_{p{\rm-var},[s,t]}\ltn\omega-\omega'\rtn_{p{\rm-var},[s,t]},
	\end{eqnarray*}
	which yields
	\begin{eqnarray*}
		\ltn y\rtn_{p{\rm-var},[s,t]} \leq M^* (t-s+\ltn\omega\rtn_{p{\rm-var},[s,t]})\|y\|_{p{\rm-var},[s,t]}+M^*\|x'\|_{p{\rm-var},[t_0,t_0+T]}\ltn\omega-\omega'\rtn_{p{\rm-var},[s,t]} 
	\end{eqnarray*}
	By applying \cite[Corollary 3.5]{congduchong17}, we obtain
	\begin{eqnarray*}
		\ltn y\rtn_{p{\rm-var},[t_0,t_0+T]} &\leq& (|y(t_0)|+M^*\|x'\|_{p{\rm-var},[t_0,t_0+T]}\ltn\omega-\omega'\rtn_{p{\rm-var},[t_0,t_0+T]}) e^{D_2(T^p+\ltn\omega\rtn^p_{p{\rm-var},[t_0,t_0+T]})}\\
		&\leq& \left( |y(t_0)|+D_1M^* \ltn\omega-\omega'\rtn_{p{\rm-var},[t_0,t_0+T]}\right)e^{D_2(T^p+\ltn\omega\rtn^p_{p{\rm-var},[t_0,t_0+T]})}\\
		&\leq & D_3( |x_0 -x_0'|+ \ltn \omega\rtn_{p{\rm-var},[a,a']}+ \ltn\omega-\omega'\rtn_{p{\rm-var},[t_0,t_0+T]}),
	\end{eqnarray*}
	where $D_2, D_3$ are constant depending on $x$ and $M^*$. Therefore
	\begin{eqnarray*}
		\| y\|_{p{\rm-var},[t_0,t_0+T]} &\leq& |y(t_0)| +  \ltn y\rtn_{p{\rm-var},[t_0,t_0+T]}\\
		&\leq& D_4(|x_0 -x_0'|+ \ltn \omega\rtn_{p{\rm-var},[a,a']}+ \ltn\omega-\omega'\rtn_{p{\rm-var},[t_0,t_0+T]}),
	\end{eqnarray*}
	with some constant $D_4$, which proves the continuity of $X$.
	%\end{comment}
\end{proof}
%=================
\subsubsection*{Young integral on infinite domain} Consider $f:\R^+\rightarrow \R$  such that $\int_a^b f(s)d\omega(s)$ exists for all $a<b\in \R^+$. Fix $t_0\geq 0$, we define $\int_{t_0}^\infty f(s)d\omega(s)$ as the limit $\displaystyle\lim_{t\to\infty}\int_{t_0}^{t}f(s)d\omega(s)$ if the limit exists and is finite. In this case, $$\int_0^{t_0}f(s)d\omega(s) = \int_0^{\infty}f(s)d\omega(s)-\int^\infty_{t_0}f(s)d\omega(s)$$
By assumption (${\textbf H}_3$) the sequence $\Big\{\frac{ \ltn\omega\rtn_{p{\rm-var},[k,k+1]}}{k},\; k\geq 1 \Big \}$ is bounded.
\begin{lemma}\label{lemma11}
	Consider $G(t) = \int_0^tg(s)d\omega(s)$, where $g$  is of bounded  $q-$variation function on every compact interval. If $\chi(g(t)),\;\chi(\ltn g\rtn_{q{\rm-var},[t,t+1]})\leq \lambda\in [0,+\infty)$ then 
	\[ 
	\chi(G(t)), \chi(\ltn G\rtn_{q{\rm-var},[t,t+1]})\leq \lambda.
	\]
\end{lemma}
%================================
\allowdisplaybreaks 
\begin{proof}
Since $\chi(g(t)),\;\chi(\ltn g\rtn_{q{\rm-var},[t,t+1]})\leq \lambda$, for any $\varepsilon>0$, there exists $D_1=D_1(\varepsilon)$ such that $|g(s)|\leq D_1e^{(\lambda+\varepsilon/2)s}$, $\ltn g\rtn_{q{\rm-var},[s,s+1]}\leq D_1e^{(\lambda+\varepsilon/2)s}$ for all $s>0$.
Then
	\begin{eqnarray*}
			|G(t)|&\leq &\sum_{k=0}^{\lfloor t\rfloor-1}\Big|\int_k^{k+1}g(s)d\omega(s)\Big|+ \Big|\int_{\lfloor t\rfloor}^tg(s)d\omega(s)\Big| \\
			&\leq &K \sum_{k=0}^{\lceil t\rceil-1}\ltn\omega\rtn_{p{\rm-var},[k,k+1]} (|g(k)|+\ltn g\rtn_{q{\rm-var},[k,k+1]})\\
			&\leq &2KD_1 \sum_{k=0}^{\lceil t\rceil-1}\frac{\ltn \omega\rtn_{p{\rm-var},[k,k+1]}}{k} ke^{(\lambda+\varepsilon/2)k}\\
			&\leq &2KD_1 \sup_{k\geq 1}\frac{\ltn \omega\rtn_{p{\rm-var},[k,k+1]}}{k}    (t+1)^2e^{t(\lambda+\varepsilon/2)}\\
			&\leq & D_2e^{(\lambda+\varepsilon)t},
	\end{eqnarray*}
	where $D_2$ is a generic constant depends on $\varepsilon$.
	This yields $\chi(G(t))\leq \lambda$.
	
Next, fix $t_0> 0$ then $[t_0,t_0+1]\subset [n_0,n_0+2]$ with some $n_0\in \N$. For each $s,t\in [t_0,t_0+1]$ we have 
		\begin{eqnarray*}
			|G(t)-G(s)|&\leq & K\ltn\omega\rtn_{p{\rm-var},[s,t]} (\|g\|_{\infty,[s,t]}+\ltn g\rtn_{q{\rm-var},[s,t]})\\
			&\leq & 2KD_1\ltn\omega\rtn_{p{\rm-var},[s,t]} e^{(\lambda+\varepsilon/2)(t_0+1)}.
		\end{eqnarray*}
Hence
\begin{eqnarray*}
			\ltn G\rtn_{q{\rm-var},[t_0,t_0+1]}&\leq &2 KD_1 \ltn\omega\rtn_{p{\rm-var},[t_0,t_0+1]} e^{(\lambda+\varepsilon/2)(t_0+1)}\\
			&\leq &2.2^{\frac{p-1}{p}}KD_1 \frac{(\ltn \omega\rtn_{p{\rm-var},[n_0,n_0+1]}+\ltn \omega\rtn_{p{\rm-var},[n_0+1,n_0+2]})}{t_0+1}(t_0+1) e^{(\lambda+\varepsilon/2)(t_0+1)}\\
			&\leq & D_2 e^{(\lambda+\varepsilon)t_0},
		\end{eqnarray*}
which implies $\chi(\ltn G\rtn_{q{\rm-var},[t,t+1]})\leq \lambda$. 
\end{proof}
%====================================================================
\begin{lemma}\label{lemma12}
	Let $g$ be of bounded  $q-$variation function on every compact interval,  
satisfying\\ $\chi(g(t)),\;\chi(\ltn g\rtn_{q{\rm-var},[t,t+1]})\leq -\lambda\in (-\infty,0)$. Then the integral $G(t) := \int_t^\infty g(s)d\omega(s)$ exists for all $t\in \R^+$ and 
\[
\chi(G(t)), \chi(\ltn G\rtn_{q{\rm-var},[t,t+1]})\leq -\lambda.
\]
\end{lemma}

%================================
\begin{proof}
	For each $\varepsilon>0$ such that $2\varepsilon <\lambda$, there exists a constant $D_1$ such that 
	$$|g(s)|\leq D_1e^{(-\lambda+\varepsilon)s},\;\;\ltn g\rtn_{q{\rm-var},[s,s+1]}\leq D_1e^{(-\lambda+\varepsilon)s}.$$
Now fix $t_0 \geq 0$ we first prove the existence and finiteness of $\lim\limits_{n\in\N\atop n\to \infty}\int_{t_0}^ng(s)d\omega(s)$. For all $n<m\in \N $ we have	
\begin{eqnarray*}
		\Big|\int_{n}^{m}g(s)d\omega(s)\Big|&\leq &  \sum_{k=n}^{m-1}\Big|\int_{k}^{k+1}g(s)d\omega(s)\Big| \\
		&\leq &K \sum_{k=n}^{m-1} \ltn \omega\rtn_{p{\rm-var},[k,k+1]}(|g(k)|+\ltn g\rtn_{q{\rm-var},[k,k+1]})\\
		&\leq & 2KD_1\sum_{k=n}^{m-1} \ltn \omega\rtn_{p{\rm-var},[k,k+1]}e^{(-\lambda+\varepsilon)k}\\
		&\leq &2KD_1e^{(-\lambda +2\varepsilon)n} \sup_k \frac{ \ltn \omega\rtn_{p{\rm-var},[k,k+1]}}{k}\sum_{k=0}^{\infty}ke^{-\varepsilon k}  \\
		&\leq &D_2e^{(-\lambda +2\varepsilon)n}
\end{eqnarray*}
which converges to zero as $n,m \to \infty$ since $\frac{ \ltn \omega\rtn_{p{\rm-var},[k,k+1]}}{k}$ are bounded and the series $\sum_{k=0}^{\infty}ke^{-k\varepsilon/2}$ converges. Therefore $\lim\limits_{n\in\N\atop n\to \infty}\int_{t_0}^ng(s)d\omega(s)<\infty$. Moreover, for $t>t_0$  by a similar estimate we have
$$	\Big|\int_{t_0}^{t}g(s)d\omega(s)-\int_{t_0}^{{\lfloor t\rfloor}}g(s)d\omega(s)\Big|=\Big|\int_{{\lfloor t\rfloor}}^{t}g(s)d\omega(s)\Big| \to 0 $$
as $t\to \infty.$
This implies the existence of $\int_{t_0}^\infty g(s)d\omega(s) $. Moreover,  $|G(t)|\leq C(\varepsilon)e^{(-\lambda+2\varepsilon)t}$ which yields $\chi(G(t))\leq -\lambda$.\\	
	The second conclusion can be proved similarly to Lemma \ref{lemma11}.
\end{proof}
%==============================

%======================
The following lemma shows that the condition (${\textbf H}_4$) is satisfied for almost all realization $\omega$ of  a fractional Brownian motion $B^H_t(\omega)$  (see \cite{mishura} for definition and details on fractional Brownian motions).
\begin{lemma} \label{lemma2a}
	Assuming that $c_0:=\|c\|_{\infty,\R^+} < \infty$ and the integral	$$X(t,\omega) = \int_0^t c(s)dB^H_s(\omega)$$ exists for all $t\in\R^+$.
	Then
$\lim\limits_{n\to\infty\atop n\in \N} \frac{X(n,\cdot)}{n} = \lim\limits_{n\to\infty\atop n\in \N} \ \frac{\int_0^nc(s)dB^H_s}{n} = 0,\  \text{a.s.} $
\end{lemma}

\begin{proof}
Fix $T>0$, and assume that $\pi_n$ is a sequence of partition of $[0,T]$ such that $mesh(\pi_n) \to 0$ as $n\to \infty$. Denote $$X_n(t,\omega) = \sum c(t_i) (B^H_{t_{i+1}}(\omega)-B^H_{t_i}(\omega)).$$
	Then $X_n(t,\omega) \to X(t,\omega)$ as $n\to \infty$.
	It is evident that $X_n$ is a Gaussian random variable with mean zero. Since 
	\begin{eqnarray}
	E(B^H_t-B^H_s)^2&= & |t-s|^{2H}=H(2H-1)\int_s^t\int_s^t|a-b|^{2H-2}dadb\notag\\
	E(B^H_t-B^H_s)(B^H_u-B^H_v) &= &\frac{1}{2} \left[|s-u|^{2H}+|t-v|^{2H}-|t-u|^{2H}-|s-v|^{2H}\right]\notag\\
	&=&(2H-1)H\int_v^u\int_s^t|a-b|^{2H-2}dadb
	\end{eqnarray}
	for all $v<u\leq s<t$ (see \cite[p. 7-8]{mishura}), we have
	\begin{eqnarray}
	V(X_n)= EX_n^2&= &\sum_{i,j=1}^nc(t_i)c(t_j) E(\Delta B^H_{t_i}\Delta B^H_{t_j})\notag\\
	&=&(2H-1)H\sum_i c^2(t_i)\int_{t_i}^{t_{i+1}}\int_{t_i}^{t_{i+1}}|u-v|^{2H-2}dudv \notag\\
	&+& 2(2H-1)H\sum_{i<j} c^2(t_i)\int_{t_i}^{t_{i+1}}\int_{t_j}^{t_{j+1}}|u-v|^{2H-2}dudv \notag\\
	&=& D(H)\int_0^t\int_0^tc(u)c(v)|u-v|^{2H-2}dudv \leq D(H)c_0^2 t^{2H},\notag
	\end{eqnarray}
	where $D(H)$ is a generic constant depending on $H$.
	
	Since $X_n \to X$, a.s, 
	$X(t,.)$ is a  centered normal random variable with $V(X(t,\cdot))\leq D(H)c^2_0t^{2H}$. It follows that $EX(t,.)^{2k}\leq D(H,k,c_0) t^{2kH}$ with $D(H,k,c_0)$ is a constant depending on $H,k,c_0$. 
	Fix $0<\varepsilon < 1-H$ and choose $k$ large enough so that $k(1-\varepsilon-H) \geq 1$ we then have
	\begin{eqnarray*}
	\sum_{n=1}^\infty P(|\frac{X(n,\cdot)}{n}| > \frac{1}{n^\varepsilon})&\leq & \sum_{n=1}^\infty \frac{EX(n,\cdot)^{2k} }{n^{2k(1-\varepsilon)}}\\
	&\leq & D(H,k,c_0)\sum_{n=1}^\infty  \frac{1}{n^{2k(1-\varepsilon-H)}}\\
	&\leq & D(H,k,c_0)\sum_{n=1}^\infty  \frac{1}{n^2} < \infty.
	\end{eqnarray*}
	Using Borel-Caltelli lemma, we conclude that $\frac{X(n,\cdot)}{n} \to 0$ as $n\to \infty$ almost surely.
	\end{proof}

\begin{proof}[{\bf Lemma \ref{compact}}]
	The if part is obvious since it can be proved that
	\begin{eqnarray}
	\lim \limits_{\delta \to 0} m^{[a,b]}(c,\delta)  &=& 0, \\
	|m^{[a,b]}(c,\delta) - m^{[a,b]}(c^\prime,\delta)| &\leq& \ltn c -c^\prime \rtn_{\alpha,[a,b]},  
	\end{eqnarray}
	which shows the continuity of $m$ on $\cC^{0,\alpha-\rm{Hol}}(\R,\R^k)$. Hence $m$ is uniformly continuous on a compact set, which shows \eqref{mcompact1} and \eqref{mcompact2}.\\
	To be more precise, denote by $\tilde{C}$ the space $\cC^{0,\alpha-\rm{Hol}}(\R,\R^k)$. Assume that $\cH$ is compact in $\tilde{C}$, we prove that \eqref{mcompact1} and \eqref{mcompact2} are fulfilled. For each $n\in\N^*$, put
	$$G_n = \{c\in \tilde{C}\; \mid |c(n)|< n\}.$$
	Then $G_n$ is open in $\tilde{C}$. \\
	Since  $\overline{\cH} \subset \bigcup_{n=1}^{\infty}G_n$  and $G_n$ is an increasing sequence of open sets, there exists $n_0$ such that $\cH\subset G_{n_0}$, which proves \eqref{mcompact1}.\\
	To prove \eqref{mcompact2}, first note that for each $c\in \tilde{C}$ and $[a,b] \subset \R$,
	$\lim \limits_{\delta \to 0} m^{[a,b]}(c,\delta)  =0$  (see \cite[Theorem 5.31,p. 96]{friz}). 
	Secondly
	$$|m^{[a,b]}(c,\delta) - m^{[a,b]}(c^\prime,\delta)| \leq \ltn c -c^\prime \rtn_{\alpha{\rm-Hol},[a,b]}.$$
	Indeed, due to the definition of $m^{[a,b]}(c,\delta)$, there exists for any $\varepsilon>0$ two points $s_0,t_0\in [a,b]$, $0<|s_0-t_0|\leq \delta$ such that
	$$m^{[a,b]}(c,\delta) \leq \frac{|c(t_0)-c(s_0)|}{|t_0-s_0|^\alpha}+\varepsilon.$$
	On the other hand, $m^{[a,b]}(c',\delta)\geq \frac{|c'(t_0)-c'(s_0)|}{|t_0-s_0|^\alpha}$, which yields
	\begin{eqnarray*}
		m^{[a,b]}(c,\delta) -   m^{[a,b]}(c',\delta)&\leq&   \frac{|c(t_0)-c(s_0)|- | c'(t_0)-c'(s_0) |}{|t_0-s_0|^\alpha}+\varepsilon\\
		&\leq&  \frac{|c(t_0)-c(s_0) -c'(t_0)+c'(s_0) |}{|t_0-s_0|^\alpha}+\varepsilon\\
		&\leq & \ltn c-c'\rtn_{\alpha{\rm-Hol},[a,b]}+\varepsilon
	\end{eqnarray*}
	Exchanging the role of $c$ and $c'$ we obtain
	$$|m^{[a,b]}(c,\delta) -   m^{[a,b]}(c',\delta)| \leq \ltn c-c'\rtn_{\alpha{\rm-Hol},[a,b]}$$
	since $\varepsilon $ is arbitrary.\\
	We now prove the continuity of the map
	$$m^{[a,b]}(\cdot\;,\delta): (\tilde{C},d)\rightarrow \R.$$
	In fact, fix $[-n,n]$ contains $[a,b]$. For each $c_0\in \tilde{C}$ and $\varepsilon\in (0,1)$ choose $\eta=\varepsilon/2^n$.  If $d(c,c_0)<\eta$ we have $\|c-c_0\|_{\alpha,[-n,n]}\wedge 1\leq 2^nd(c,c_0)< \varepsilon$.
	Therefore
	$$|m^{[a,b]}(c,\delta) -   m^{[a,b]}(c',\delta)|\leq  \ltn c-c'\rtn_{\alpha,[-n,n]}\leq \varepsilon.$$ 
	\\
	Next, fix $\varepsilon>0$ and  define the set $$K_{\delta}:=\{c\in \bar{A}\;\mid\; m^{[a,b]}(c,\delta) \geq \varepsilon\}.$$
	Then $K_{\delta}$ is closed for all $\delta$. Due to the fact that $\lim \limits_{\delta \to 0} m^{[a,b]}(c,\delta)  =0$  for all $c\in \tilde{C}$, we have $\displaystyle\bigcap_{\delta>0}K_{\delta}=\emptyset$. Then there exists $\delta=\delta(\varepsilon)>0$ such that $K_{\delta}=\emptyset$, which proves \eqref{mcompact2}.
	\medskip
	
	For the  ''only if'' part, assume that \eqref{mcompact1} and \eqref{mcompact2} hold, we are going to prove the compactness of $\bar{\cH}$. Since $\tilde{C}$ is a complete metric space, it suffices to prove that every sequence $\{c_n\}_{n=1}^\infty \subset \cH$ has a convergent subsequence. Following the arguments of \cite[Theorem 4.9, p. 63]{karatzas} line by line, we can construct a convergent subsequence $\{\tilde{c}_n\}_{n=1}^\infty$ by the ''diagonal sequence'' such that $\tilde{c}_n(r) \to c(r)$ as $n\to \infty$ for any rational number $r \in \Q$. With \eqref{mcompact1} and \eqref{mcompact2},  $\cH$  satisfies the condition in  \cite[Theorem 4.9, p. 63]{karatzas}, hence $\tilde{c}_n$ converge uniformly to a continuous function $c$ in every $[a,b]\subset\R$.\\
	Fix $[a,b]$, by \eqref{mcompact2} for each $\varepsilon>0$ there exist $\delta_0>0$ such that if $\delta\leq\delta_0$, $\sup_{s,t\in[a,b]\atop |s-t|\leq \delta}\frac{|\tilde{c}_n(t)-\tilde{c}_n(s)|}{|t-s|^{\alpha}}\leq \varepsilon$ for all $n$. Hence
	$$\sup_{s,t\in[a,b]\atop |s-t|\leq \delta}\frac{|c(t)-c(s)|}{|t-s|^{\alpha}}\leq \varepsilon,$$
	thus $c\in \tilde{C}$.
	Finally, we prove that $\tilde{c}_n$ converge to $c$ in the H\"older seminorm  on every compact interval $[a,b]$. Namely, 
	with $\varepsilon, \delta_0$ given, there exist $n_0$ such that for all $n\geq n_0$, $\|\tilde{c}_n-c\|_{\infty,[a,b]}\leq \delta^{\alpha}_0\varepsilon$. Then for $n\geq n_0$
	\begin{eqnarray*}
		\sup_{s,t\in[a,b]}\frac{|(\tilde{c}_n-c)(t)-(\tilde{c}_n-c)(s)|}{|t-s|^\alpha} & \leq & \sup_{s,t\in[a,b]\atop |t-s|\leq \delta_0}\frac{|(\tilde{c}_n-c)(t)-(\tilde{c}_n-c)(s)|}{|t-s|^\alpha}\\
		&&+\sup_{s,t\in[a,b]\atop |t-s|\geq \delta_0}\frac{|(\tilde{c}_n-c)(t)-(\tilde{c}_n-c)(s)|}{|t-s|^\alpha}\\
		&\leq & m^{[a,b]}(\tilde{c}_n,\delta_0)+m^{[a,b]}(c,\delta_0)+ \frac{2}{\delta^{\alpha}_0}\|\tilde{c}_n-c\|_{\infty,[a,b]}\\
		&\leq & 4\varepsilon,
	\end{eqnarray*}
	which implies that $\ltn\tilde{c}_n-c\rtn_{\alpha{\rm -Hol},[a,b]}$ converges to $0$ as $n\to\infty$. This completes the proof.
\end{proof}

%%%%%%%%%%%%%%%%%%%%%%%%%%%%%%%%%%%%%%%%%%%%%%%%%%%%%%

\end{document}